\documentclass[12pt]{amsart}
\usepackage[top=1in,left=1.25in,right=1.25in,bottom=1in]{geometry} 
\usepackage{geometry}

\usepackage{amsmath}
\usepackage{amsthm}
\usepackage{amssymb}
\usepackage{lscape,xcolor}
\usepackage{graphicx}
\usepackage{mathrsfs}
\usepackage{stmaryrd}
\usepackage{verbatim}
\usepackage{rotating}
\usepackage{tikz-cd}
\usepackage{amsrefs}
\usepackage{hyperref}
\usepackage{euscript}
\usepackage[colorinlistoftodos]{todonotes}

\usepackage[all]{xy}
\usepackage{cleveref} 
\usepackage{leftidx}
\usepackage[shortlabels]{enumitem}

\usepackage{xcolor}
\definecolor{seagreen}{RGB}{46,139,87}
\definecolor{maroon}{RGB}{128,0,0}
\definecolor{darkviolet}{RGB}{148,0,211}
\definecolor{twelve}{RGB}{100,100,170}
\definecolor{thirteen}{RGB}{100,150,50}
\definecolor{fourteen}{RGB}{200,0,0}
\definecolor{fifteen}{RGB}{0,200,0}
\definecolor{sixteen}{RGB}{0,0,200}
\definecolor{seventeen}{RGB}{200,0,200}
\definecolor{eighteen}{RGB}{0,200,200}

\allowdisplaybreaks[1]

\newcommand{\mmod}{\! \sslash \!}

\newcommand{\mr}[1]{\mathrm{#1}}

\newcommand{\br}[1]{\overline{#1}}

\newcommand{\N}{\mathbb{N}}
\newcommand{\Z}{\mathbb{Z}}

\newcommand{\Q}{\mathbb{Q}}

\newcommand{\F}{\mathbb{F}}

\newcommand{\PP}{\mathbb{P}}
\newcommand{\TT}{\mathbb{T}}

\newcommand{\Tmf}{\mathrm{Tmf}}
\newcommand{\tmf}{\mathrm{tmf}}

\def \HF2{\mr{H}\F_2}

\DeclareMathOperator{\Hom}{Hom}
\DeclareMathOperator{\Ext}{Ext}

\DeclareMathOperator{\Map}{Map}
\DeclareMathOperator{\Maps}{Maps}
\DeclareMathOperator*{\holim}{holim}

\DeclareMathOperator*{\Tot}{Tot}
\DeclareMathOperator{\Spf}{Spf}

\def \AA0{\br{A \mmod A(0)}_*}
\def \AA2{A\mmod A(2)_*}
\def \AE2{(A\mmod E(2))_*}
\renewcommand{\AE}[1]{(A\mmod E(#1))_*}

\def \E2E1{(E(2)\mmod E(1))_*}

\newcommand{\Alg}{\mathsf{Alg}}

\newcommand{\CAlg}{\mathsf{CAlg}}
\newcommand{\Sp}{\mathsf{Sp}}
\newcommand{\Comod}{\mathsf{Comod}}

\newcommand{\Mord}{\mathscr{M}_{\mathrm{ell}}^{ord}}

\newcommand{\coker}{\mathrm{coker}}

 \newtheorem{thm}[equation]{Theorem}
 \newtheorem{cor}[equation]{Corollary}
 \newtheorem{lem}[equation]{Lemma}
 \newtheorem{prop}[equation]{Proposition}

 \newtheorem*{thm*}{Theorem}
 \newtheorem*{cor*}{Corollary}
 \newtheorem*{lem*}{Lemma}
 \newtheorem*{prop*}{Proposition}
  \newtheorem*{not*}{Notation}
  
  \newtheorem{thmalph}{Theorem}

 \theoremstyle{definition}
 \newtheorem{defn}[equation]{Definition}
 \newtheorem{ex}[equation]{Example}
 
 \newtheorem{rmk}[equation]{Remark}

\newtheorem*{defn*}{Definition}
\newtheorem*{ex*}{Example}
\newtheorem*{exs*}{Examples}
\newtheorem*{rmk*}{Remark}
\newtheorem*{claim*}{Claim}

\numberwithin{equation}{section}
\numberwithin{figure}{section}

\newcommand{\Cts}{\Maps_{\mathrm{cts}}}

\DeclareMathOperator{\Tor}{Tor}

\newcommand{\Witt}{\mathbb{W}}
\newcommand{\Mod}{\mathsf{Mod}}

\newcommand{\Mtwo}{\mathcal{M}_{\mathrm{pair}}}
\newcommand{\Rtwo}{R_{\mathrm{pair}}}

\newcommand{\pushoutcorner}[1][dr]{\save*!/#1+1.2pc/#1:(1,-1)@^{|-}\restore}

\newcommand{\surj}{\twoheadrightarrow}

\newcommand{\sma}{\wedge}

\newcommand{\ol}[1]{\overline{#1}}

\newcommand{\wh}[1]{\widehat{#1}}

\title{On the \texorpdfstring{$K(1)$}{K(1)}-local homotopy of \texorpdfstring{$\tmf\wedge \tmf$}{tmf smash tmf}}
\author{Dominic Leon Culver}\address{University of Illinois, Urbana-Champaign}\email{dculver@illinois.edu}
\author{Paul VanKoughnett}
\thanks{The second author was supported by the National Science Foundation Grant No. 1440140, while he was in residence at the Mathematical Sciences Research Institute in Berkeley, California, during the spring semester of 2019, as well as under Grant No. 1714273.}
\address{Purdue University}
\email{pvankoug@purdue.edu}

\begin{document}

\begin{abstract}
	As a step towards understanding the $\tmf$-based Adams spectral sequence, we compute the $K(1)$-local homotopy of $\tmf \sma \tmf$, using a small presentation of $L_{K(1)}\tmf$ due to Hopkins. We also describe the $K(1)$-local $\tmf$-based Adams spectral sequence.
\end{abstract}

\maketitle

\tableofcontents

\section{Introduction}

This paper calculates the $K(1)$-local homotopy of $\tmf\wedge \tmf$. The motivation behind this traces back to Mahowald's work on $bo$-resolutions. In his seminal papers on the subject (\cite{mahowald1981}, \cite{LellmanMahowald}), Mahowald was able to use the $bo$-based Adams spectral sequence
\begin{enumerate}
	\item to prove the height 1 telescope conjecture at the prime $p=2$,
	\item and, with Wolfgang Lellmann, to exhibit the $bo$-based Adams spectral sequence as a viable tool for computations.
\end{enumerate}
An initial difficulty with this spectral sequence is the fact that $bo_*bo$ does not satisfy Adams' flatness assumption, resulting in the $E_2$-term not having a description in terms of $\Ext$. One can still work with the spectral sequence, but one has to understand both the algebra $bo_*bo$ and the homotopy theory of $bo$-modules extremely well, and Mahowald's breakthrough decomposition of $bo \sma bo$ in terms of Brown-Gitler spectra satisfied both goals.

Mahowald later initiated the study of resolutions over $\tmf$, first known as $eo_2$. Early work on this was done by Mahowald and Rezk in \cite{MahowaldRezk}, and then developed further in the work of Behrens-Ormsby-Stapleton-Stojanoska in \cite{BOSS}. Again, to work with the $\tmf$-based Adams spectral sequence, one first needs to understand of the  homotopy groups $\pi_*(\tmf\wedge \tmf)$. This computation was seriously studied in \cite{BOSS} at the prime 2, and at the prime 3 is ongoing work of the first author and Vesna Stojanoska. 

Behrens-Ormsby-Stapleton-Stojanoska take a number of approaches to $\tmf_*\tmf$:
\begin{enumerate}
	\item The \textit{rational homotopy} $\tmf_*\tmf \otimes \Q$, can be described as a ring of rational, 2-variable modular forms.
	\item The \textit{$K(2)$-local homotopy} $\pi_*L_{K(2)}(\tmf \sma \tmf)$ can be described in terms of Morava $E$-theory using the methods of \cite{GHMR}. To be precise, one has
	\[
L_{K(2)}(\tmf\wedge \tmf)\simeq \left(\Map^c(\mathbb{S}_2/G_{24}, \overline{E_2})^{hG_{24}}\right)^{hGal}.
\]
	\item Using a change of rings isomorphism, one can write the \textit{classical Adams spectral sequence} as
\[
E_2 = \Ext_{A_*}(H_*\tmf\wedge \tmf) \cong \Ext_{A(2)_*}(A\mmod A(2)_*) \implies \pi_*\tmf\wedge \tmf.
\]
However, the $E_2$-term is rather difficult to calculate since the algebra $A\mmod A(2)_*$ is very complicated. Indeed, a full computation of the Adams $E_2$-term has yet to be done. The approach via the Adams spectral sequence is further complicated by the presence of differentials. Such differentials were first discovered in \cite{MahowaldRezk}, and even more were found in \cite{BOSS}. 
\end{enumerate}

Chromatic homotopy theory in principle allows the reassembly of $\tmf \sma \tmf$ from its rationalization, $K(1)$-localizations at all primes, and $K(2)$-localizations at all primes. In this paper, we approach the as-yet-unstudied chromatic layer, giving a complete description of $L_{K(1)}(\tmf \sma \tmf)$. Our main tool is a construction due to Hopkins of $K(1)$-local $\tmf$ as a small cell complex in $K(1)$-local $E_\infty$-rings \cite{K1localrings}.

Let us briefly mention some intuition and notation before stating the main result. First, the ring $\pi_*L_{K(1)}\tmf$ is essentially a graded version of the ring of functions on the $p$-complete moduli stack $\Mord$ of ordinary, generalized elliptic curves \cite{Laures2004}. At small primes $p \le 5$, we have
\[	\pi_0L_{K(1)}\tmf = \Z_p[j^{-1}]^\wedge_p,	\]
where $j^{-1}$ is the inverse of the modular $j$-invariant. (Note that, at these primes, $\Mord$ includes the point $j = \infty$, corresponding to the nodal cubic, but not the point $j = 0$, which is supersingular for $p \le 5$.) If one writes $KO$ for 2-complete real $K$-theory if $p = 2$, or the $p$-complete Adams summand for $p > 2$, the formula in all degrees (still for $p \le 5$) becomes
\[	\pi_*L_{K(1)}\tmf = (KO_*[j^{-1}])^\wedge_p.	\]
This has $p$-torsion just at $p = 2$.

Second, the 0th homotopy group of a $K(1)$-local $E_\infty$-ring is naturally a $\theta$-algebra, bearing an algebraic structure studied extensively by Bousfield \cite{Bousfield96} and described briefly in our \Cref{subsec: theta-algebras}. We write $\TT(x)$ for the free $\theta$-algebra on a generator $x$; by a theorem of Bousfield, as a ring, $\TT(x)$ is polynomial on $x$, $\theta(x)$, $\theta^2(x)$, and so on.

We can now state the main result.

\begin{thmalph}\label{thm: mainthm}
At primes $p \le 5$,
	\[
	\pi_*L_{K(1)}(\tmf\wedge \tmf)\cong \left(KO_*[j^{-1}, \ol{j^{-1}}] \otimes \TT(\lambda)/(\psi^p(\lambda) - \lambda - j^{-1} + \ol{j^{-1}})\right)^\wedge_p.
	\]
\end{thmalph}

Given this, the last remaining obstacle to a chromatic understanding of $\tmf_*\tmf$ is a calculation of the transchromatic map 
\[	L_{K(1)}(\tmf \sma \tmf) \to L_{K(1)}L_{K(2)}(\tmf\wedge \tmf).	\]
We hope to study this in future work.

Let us describe a few consequences of this result. One is a computation of the $K(1)$-local Adams spectral sequence based on $\tmf$.

\begin{thmalph}\label{thm: ASS}
For any spectrum $X$, there is a conditionally convergent spectral sequence
\[	E_2 = \Ext_{\pi_*L_{K(1)}(\tmf\wedge \tmf)}(\pi_*L_{K(1)}\tmf, \pi_*L_{K(1)}(\tmf \sma X)) \Rightarrow \pi_*L_{K(1)}X.	\]
When $X$ is the sphere, the $E_2$ page of this spectral sequence is isomorphic to
\[\begin{split}
	\Ext_{\pi_*L_{K(1)}(\tmf\wedge \tmf)}(\pi_*L_{K(1)}\tmf, \pi_*L_{K(1)}\tmf) \cong \Ext_{\pi_*L_{K(1)}(KO \sma KO)}(KO_*, KO_*) \\
	\cong H^*_{cts}(\Z_p^\times/\mu, KO_*),
\end{split}\]
where $\mu$ is the maximal finite subgroup of $\Z_p^\times$.
\end{thmalph}

In particular, the spectral sequence for the sphere vanishes at $E_2$ above cohomological degree 1, and so collapses immediately. While the $K(1)$-local $\tmf$-based Adams spectral sequence is thus uninteresting, one obtains some nontrivial information about the global $\tmf$-based Adams spectral sequence, namely that its $v_1$-periodic classes occur only on the 0 and 1 lines.

To put these results into perspective, it helps to return to $bo$. $K(1)$-locally, $bo$ is the same as $KO$, and its $K(1)$-local co-operations algebra is simply:
\[	\pi_*L_{K(1)}(bo \sma bo) = \pi_*L_{K(1)}(KO \sma KO) = KO_* \otimes \Cts(\Z_p^\times/\mu, \Z_p).	\]
As $bo \sma bo$ is $E_\infty$, this ring has an alternative $\theta$-algebraic description, namely
\[	\pi_*L_{K(1)}(bo \sma bo) = KO_* \otimes \TT(b)/(\psi^p(b) - b).	\]
Here $b$ is an explicit choice of group isomorphism $\Z_p^\times/\mu \stackrel{\cong}{\to} \Z_p$, and the single relation expands to
\[	p\theta(b) = b - b^p,	\]
a relation between $b$ and $\theta(b)$. In the formula of \Cref{thm: mainthm}, the modular forms $j^{-1}$, $\ol{j^{-1}}$ also satisfy $\theta$-algebra relations forced on them by number theory, and one obtains a relation between $\lambda$, $\theta(\lambda)$, and $\theta^2(\lambda)$, a sort of second-order version of the $bo$ calculation.

It is also worth noting that, for the sake of calculating Adams spectral sequences, one is interested in the coalgebra of $bo_*bo$ as much as its algebra -- and the original, non-$\theta$-algebraic calculation
\[	\pi_*L_{K(1)}(bo \sma bo) = KO_* \otimes \Cts(\Z_p^\times/\mu, \Z_p)	\]
is actually better suited for this purpose. It is this realization, and a search for an analogue for $\tmf$, that eventually led to the proof of \Cref{thm: ASS}.

As a final remark, our calculation also doubles as a calculation of a purely number-theoretic object. Namely, consider the moduli problem $\Mtwo$ over $\Spf \Z_p$ that sends a $p$-complete ring $R$ to the groupoid of data 
\[	(E, E', \phi:E \stackrel{\sim}{\to} E'),	\]
where $E$ and $E'$ are ordinary generalized elliptic curves over $R$ and $\phi$ is an isomorphism of their formal groups. Just as the structure sheaf of the moduli of generalized elliptic curves extends to a locally even periodic sheaf of $E_\infty$ ring spectra whose global sections are (the nonconnective) $\Tmf$ \cite{Behrensconstruction, GoerssTMF}, there is such a sheaf on $\Mtwo$ whose global sections are $L_{K(1)}(\tmf \sma \tmf)$. Moreover, $\Mtwo$ is an affine scheme in the case $p>2$, and has a double cover by an affine scheme in the case $p = 2$. In both cases, its ring of global functions $\Rtwo$ is exactly $\pi_0L_{K(1)}(\tmf \sma \tmf)$. We can think of this ring as a ring of ``ordinary 2-variable $p$-adic modular functions''. As examples of ordinary 2-variable $p$-adic modular functions, we have the functions
\[	j^{-1}: (E, E', \phi) \mapsto j^{-1}(E), \quad \overline{j^{-1}}: (E, E', \phi) \mapsto j^{-1}(E').	\]
Of course, these examples are somewhat trivial because they are really 1-variable modular functions. The results of this paper tell us that, \textit{as a $\theta$-algebra}, $\Rtwo$ is generated over these 1-variable functions by a single other generator. This generator is explicitly given as the generator $\lambda$ described in \Cref{rmk: generator lambda}.
%
%

\subsection{Outline of the paper}

This paper is almost entirely set inside the $K(1)$-local category. This leads to some unusual choices about notation, for the sake of which we encourage even the expert reader to take a look at \Cref{sec: notation} below. In \Cref{sec: completehopfalgebroids}, we give some background information about $K(1)$-local homotopy theory, in particular reviewing the relevant notion of completeness and associated issues of homological algebra. Building on \cite{HoveyStrickland}, \cite{HoveySS}, \cite{Baker}, and \cite{BarthelHeard}, we set up some fundamental tools, such as a relative K\"unneth formula, a change of rings theorem, and the theory of $K(1)$-local Adams spectral sequences, that we will use later on.

In \Cref{sec: Cone zeta}, we study the $E_\infty$ cone on the class $\zeta \in \pi_{-1}L_{K(1)}S$, called $T_\zeta$ by Hopkins. This object was used in \cite{K1localrings} and \cite{Laures2004} as a partial version of $\tmf$, and the results in this section can mostly be found in those papers. However, in the process of reading those papers, the authors found some problems with the calculation of $\pi_*T_\zeta$ (see \Cref{rmk: Hopkins mistake}). Part of our motivation in writing down this calculation in detail is to fill these gaps.

In Section \ref{sec: coops for cone on zeta}, we compute the cooperations algebra $\pi_*L_{K(1)}(T_\zeta \sma T_\zeta)$, which is an approximation to $\pi_*L_{K(1)}(\tmf \sma \tmf)$.

In \Cref{sec: tmf}, we return to the work of Hopkins and Laures to review their construction of $L_{K(1)}\tmf$. Again, the material in this section can be found in \cite{K1localrings} or \cite{Laures2004}, but we include for the reader's convenience.

In \Cref{sec: coops for tmf}, we compute the $K(1)$-local co-operations algebra for $\tmf$, and prove Theorems \ref{thm: mainthm} and \ref{thm: ASS}.


We have also included an appendix containing technical information about $\theta$-algebras and $\lambda$-rings.

\subsection{Notation and conventions}\label{sec: notation}

\textit{The rest of this paper takes place inside the $K(1)$-local category, at a fixed prime $p \le 5$. To avoid notational clutter, we adopt a blanket convention that all objects are implicitly $K(1)$-localized and/or $p$-completed, unless it is explicitly stated otherwise.} To be precise, this includes the following conventions for algebra:
\begin{itemize}
	\item All rings are implicitly $L$-completed with respect to the prime $p$ (see \Cref{subsec: completeness}, and note that the $L$-completion agrees with the ordinary $p$-completion when the ring is torsion-free). For example, by $\Z_p[j^{-1}]$ we really mean the completed polynomial algebra
	\[	\Z_p[j^{-1}]^{\wedge}_p = \left\{ \sum_{n \ge 0} a_n j^{-n}: |a_n|_p \to 0 \text{ as }n \to \infty\right\}.	\]
	\item By $\otimes$ we mean the $L$-completed tensor product (see \Cref{subsec: completeness}).
	\item If $R_*$ is an $L$-complete ring, then $\Mod^\wedge_{R_*}$ is the category of $L$-complete $R_*$-modules and $\CAlg^\wedge_{R_*}$ the category of $L$-complete commutative $R_*$-algebras. If $(R_*, \Gamma_*)$ is an $L$-complete Hopf algebroid, then $\Comod^\wedge_{\Gamma_*}$ is its category of $L$-complete comodules (see \Cref{subsec: Lcomplete Hopf algebroids}).
	\item $\Ext_{\Gamma_*}$ is the relative Ext functor for comodules defined in \Cref{defn: relative ext}.
	\item $\TT(x_1, \dotsc, x_n)$ is the free $p$-complete $\theta$-algebra on the generators $x_1, \dotsc, x_n$ (see \Cref{thm: free theta-algebra}).
\end{itemize}
It includes the following conventions for topology:
\begin{itemize}
	\item All smash products are implicitly $K(1)$-localized.
	\item $\Sp$ is the category of $K(1)$-local spectra, and $\CAlg$ is the category of $K(1)$-local $E_\infty$-algebras.
	\item $\PP(X)$ is the free $K(1)$-local $E_\infty$-algebra on a spectrum $X$.
\end{itemize}

We will also employ the following notation:
\begin{itemize}
	\item $\mu$ is the maximal finite subgroup of $\Z_p^\times$, so $\mu \cong C_2$ if $p = 2$ or $C_{p-1}$ if $p$ is odd, and $\Z_p^\times/\mu \cong \Z_p$.
	\item $\omega$ is a generator of $\mu$ (so $\omega = -1$ at $p=2$).
	\item $g$ is a fixed element of $\Z_p^\times$ mapping to a topological generator of $\Z_p^\times/\mu$. When $p > 2$, we take $g$ to itself be a topological generator of $\Z_p^\times$.
	\item $K$ is $p$-completed complex $K$-theory, and $\tmf$ is $K(1)$-local $\tmf$. $KO$ is (2-complete) $KO$ if $p=2$, or the ($p$-complete) Adams summand if $p$ is odd.
\end{itemize}

\begin{rmk}[Restrictions on $p$]
Unless otherwise stated, the results of this paper are valid only at $p = 2$, 3, and 5. This is primarily a matter of convenience: at these primes, there is a unique supersingular $j$-invariant congruent to 0 mod $p$, which implies that $\pi_0L_{K(1)}\tmf$ is a $p$-complete polynomial in the generator $j^{-1}$. At larger primes, $\pi_0L_{K(1)}\tmf$ is the $p$-complete ring of functions on
\[	\mathbb{P}^1_{\Z_p} - \{\text{supersingular }j\text{-invariants}\},	\]
which grows more complicated as the number of supersingular $j$-invariants increases, though presumably not in an essential way.

Our restriction on $p$ is also a matter of interest: it is only at $p = 2$ and 3 that the homotopy groups of the unlocalized spectrum $\tmf$ has torsion; at larger primes $\tmf_*$ is just the ring of level 1 modular forms.

The reader will also note that the $K(1)$-local category behaves differently at the prime 2 than at all other primes. For example, while $\pi_*\tmf$ has 2- and 3-torsion, $\pi_*L_{K(1)}\tmf$ only has torsion at the prime 2.
\end{rmk}

\subsection{Acknowledgements}

The authors spoke to many individuals whose input helped in the writing of this paper. We especially thank Tobias Barthel, Mark Behrens, Siegfred Baluyot, Drew Heard, Charles Rezk, Vesna Stojanoska, and Craig Westerland. 
%
%

\section{Complete Hopf algebroids and comodules}\label{sec: completehopfalgebroids}

One often attempts to study a $K(1)$-local spectrum $X$ through its completed $K$-homology or $KO$-homology,
\[	K_*X = L_{K(1)}(K \wedge X)\text{ and }KO_*X = L_{K(1)}(KO \wedge X).	\]
These are not just graded abelian groups, but satisfy a condition known since \cite{HoveyStrickland} as $L$-completeness. In \Cref{subsec: completeness}, we review the definition of $L$-completeness and some basic properties of the $L$-complete category. Next, in \Cref{subsec: pro-freeness}, we review the important technical notion of pro-freeness, which is to be the appropriate replacement for flatness in the $L$-complete setting. As we have to deal with some relative tensor products of $K(1)$-local ring spectra, we need a relative definition of pro-freeness that is more general than that used by other authors, e.~g.~\cite{HoveySS}. We use this definition to give a K\"unneth formula for relative tensor products in which one of the modules is pro-free. In \Cref{subsec: Lcomplete Hopf algebroids}, we discuss homological algebra over $L$-complete Hopf algebroids, a concept originally due to Baker \cite{Baker}, and conclude with an examination of the $K(1)$-local Adams spectral sequence. Finally, in \Cref{subsec: Hopf algebroids K KO}, we give the classical examples of the Hopf algebroids for $K$ and $KO$, and describe their categories of comodules.

The results of this section should be compared with Barthel-Heard's work on the $K(n)$-local $E_n$-based Adams spectral sequence \cite{BarthelHeard}. While we ultimately want to write down $K(1)$-local Adams spectral sequences over more general bases than $K$ itself, the work involved is substantially simplified by certain convenient features of height 1, mostly boiling down to the fact that direct sums of $L$-complete $\Z_p$-modules are exact -- the analogue of which is not true at higher heights \cite[Proposition 1.9]{HoveySS}. The reader who wishes to do similar work at higher heights should therefore proceed with caution.

\subsection{Background on \texorpdfstring{$L$-completeness}{L-completeness}}\label{subsec: completeness}

In the category $\Sp$ of $K(1)$-local spectra, there is a well-known equivalence (\cite[Proposition 7.10]{HoveyStrickland}) 
\[
X\simeq \holim_{i} X\wedge S/p^i
\]
Replacing $X$ by the $K(1)$-local smash product $K\wedge X$, we have an equivalence
\[
K\wedge X\simeq \holim K\wedge X\wedge S/p^i.
\]
This shows that $K_*X$ is derived complete, in a sense we now make precise. 

We can regard $p$-completion as an endofunctor of the category of abelian groups. This functor is neither left nor right exact. However, it still has left derived functors, which we write as $L_0$ and $L_1$ (the higher left derived functors vanish in this case). Since $p$-completion is not right exact, it is generally \emph{not} the case that $M^\wedge_p = L_0M$. There is, however, a canonical factorization of the completion map $M\to M^\wedge_p$:
\[
\begin{tikzcd}
	M\arrow[r] & L_0M \arrow[r, "\varepsilon_M"] &  M^\wedge_p. 
\end{tikzcd}
\]
The second map is surjective, and in fact, there is a short exact sequence \cite[Theorem A.2(b)]{HoveyStrickland}
	\begin{equation}\label{eq: L-completion Milnor sequence}
	0 \to \lim_n\!^1\Tor^\Z_1(\Z/p^n, M) \to L_0M \to M^\wedge_p \to 0.
	\end{equation}
We also have \cite[Theorem A.2(d)]{HoveyStrickland}
\[	L_0M = \Ext_{\Z}^1(\Z/p^\infty, M), \quad L_1M = \Hom_{\Z}(\Z/p^\infty, M).	\]
	
\begin{defn}
		An abelian group $A$ is \emph{$L$-complete} if the natural map $A\to L_0A$ is an isomorphism. A graded abelian group $A_*$ is \emph{$L$-complete} if it is $L$-complete in each degree.
\end{defn}

	Being $L$-complete is quite close to being $p$-complete: for example, $p$-complete modules are $L$-complete, and if $M$ is finitely generated, then $L_0M\cong M^\wedge_p$. In particular, $K_*$ and $KO_*$ are $L$-complete. More generally, for any $K(1)$-local spectrum $X$, $\pi_*X$ is $L$-complete as a graded abelian group \cite[Lemma 7.2]{HoveySS}.
	
	Write $\Mod_*^\wedge$ for the category of $L$-complete graded $\Z_p$-modules. This is an abelian subcategory of the category of graded $\Z_p$-modules which is closed under extensions. It is also closed symmetric monoidal \cite[section 1.1]{HoveySS} under the $L$-completed tensor product
	\[	
	M_* \overline{\otimes} N_* = L_0(M_* \otimes N_*).
	\]
	Following our general conventions (see \Cref{sec: notation}), we will simply write $\otimes$ for this tensor product, where this does not cause confusion. 
	
	Write $\CAlg_*^\wedge$ for the category of commutative ring objects in $\Mod_*^\wedge$. If $R_* \in \CAlg_*^\wedge$ (in particular, if $R_* = K_*$ or $KO_*$), there is an obvious abelian category of $L$-complete $R_*$-modules, which we denote $\Mod_{R_*}^\wedge$.
	
\subsection{Pro-freeness}\label{subsec: pro-freeness}
	
	\begin{defn}
	Let $R_* \in \CAlg_*^\wedge$, and let $M_* \in \Mod_{R_*}^\wedge$. Say that $M_*$ is \textbf{pro-free} if it is of the form
	\[	M_* \cong L_0F_*,	\]
	where $F_*$ is a free graded $R_*$-module. Say that a map $R_* \to S_*$ of commutative rings in $\Mod_*^\wedge$ is \textbf{pro-free} if $S_*$ is a pro-free $R_*$-module.
	\end{defn}
	
	Pro-free modules are projective in the category $\Mod_{R_*}^\wedge$. In this height 1 case, they are also flat in this category. As is shown below, this follows from the fact that direct sums in $\Mod_*^\wedge$ are exact, which is, surprisingly, not true at higher heights.
	
	\begin{lem}\label{lem: pro-free flat}
	Let $R_* \in \CAlg_*^\wedge$, and let $M_*$ be a pro-free $R_*$-module. Then $M_*$ is faithfully flat in $\Mod_{R_*}^\wedge$, that is, the functor $M_* \otimes_{R_*} \cdot$ is exact and conservative.
	\end{lem}
	
	\begin{proof}
	If $M_*$ is a pro-free $R_*$-module, it is a coproduct of (possibly shifted) copies of $R_*$ in the category $\Mod^\wedge_{R_*}$. Correspondingly, $M_* \otimes_{R_*} N_*$ is a coproduct of possibly shifted copies of $N_*$, which can be taken in $\Mod^\wedge_*$. This functor is exact because coproducts in $\Mod^\wedge_*$ are exact \cite[Proposition 1.9]{HoveySS}. Clearly, a coproduct of copies of $N_*$ is zero iff $N_*$ is zero, which together with exactness implies conservativity.
	\end{proof}
	
	\begin{lem}\label{lem: pro-free base change}
	Pro-freeness is preserved by base change: if $M_*$ is pro-free over $R_*$ and $R_* \to S_*$ is a map of rings in $\Mod_*^\wedge$, then $M_* \otimes_{R_*} S_*$ is pro-free over $S_*$.
	\end{lem}
	
	\begin{proof}
	Again, $M_*$ is a coproduct of copies of $R_*$ in the category $\Mod^\wedge_{R_*}$. The tensor product is a left adjoint, so distributes over this coproduct.
	\end{proof}
	
	\begin{lem}\label{lem: pro-free mod p}
	Suppose that $R_* \in \CAlg_*^\wedge$ and $M_* \in \Mod^\wedge_{R_*}$. Suppose also that $R_*$ is $p$-torsion-free. Then $M_*$ is pro-free over $R_*$ iff $M_*$ is $p$-torsion-free and $M_*/p$ is free over $R_*/p$.
	\end{lem}
	
	\begin{proof}
	Suppose that $M_*$ is pro-free over $R_*$, and write $M_* = L_0\left(\bigoplus_\alpha \Sigma^{n_\alpha} R_*\right)$. By the exact sequence \eqref{eq: L-completion Milnor sequence}, $M_*$ is the same as the $p$-completion of $\bigoplus_\alpha \Sigma^{n_\alpha} R_*$, and is, in particular, $p$-torsion-free. By \cite[Proposition A.4]{HoveyStrickland},
	\[	L_0\left(\bigoplus_\alpha \Sigma^{n_\alpha} R_*\right)/p = \left(\bigoplus_\alpha \Sigma^{n_\alpha} R_*\right)/p = \bigoplus_\alpha \Sigma^{n_\alpha} (R_*/p),	\]
	which is clearly free over $R_*/p$ (and flat, in particular).
	
	For the converse, suppose that $M_*$ is $L$-complete and $p$-torsion-free and $M_*/p$ is free over $R_*/p$. Again using \eqref{eq: L-completion Milnor sequence}, we see that the natural surjection $M_* \to (M_*)^\wedge_p$ is an isomorphism, so that $M_*$ is honestly $p$-complete. Choose generators for $M_*/p$ as an $R_*/p$-module, and lift them to a map
	\[	\phi:F_* \to M_*	\]
	from a free graded $R_*$-module, which is an isomorphism mod $p$. Again, we observe that $L_0(F_*) = (F_*)^\wedge_p$, that it is $p$-torsion-free, and that $L_0(F_*)/p = F_*/p$. Applying the snake lemma to the diagram of graded $\Z_p$-modules
	\[	\xymatrix{ 0 \ar[r] &  L_0(F_*) \ar[d]_{\phi^\wedge} \ar[r]^p & L_0(F_*) \ar[d]^{\phi^\wedge} \ar[r] & F_*/p \ar[d] \ar[r] & 0 \\
				0 \ar[r] & M_* \ar[r]^p & M_* \ar[r] & M_*/p \ar[r] & 0, }	\]
	we see that multiplication by $p$ is an isomorphism on $\ker(\phi^\wedge)$ and $\coker(\phi^\wedge)$. Both of these are $L$-complete graded $\Z_p$-modules, and this implies that they are zero, by \cite[Theorem A.6(d,e)]{HoveyStrickland}.	
	\end{proof}
	
	\begin{lem}\label{lem: pro-free wedge of spheres}
	Let $R$ be a homotopy commutative $K(1)$-local ring spectrum, and let $M$ be a $K(1)$-local $R$-module. Then $M_*$ is pro-free over $R_*$ if and only if there is an equivalence of $K(1)$-local $R$-modules,
	\[	M \simeq \bigvee \Sigma^{n_\alpha}R.	\]
	\end{lem}
	
	(Here, as always, the coproduct is taken in the $K(1)$-local category).
	
	\begin{proof}
	Suppose that $M_*$ is pro-free over $R_*$. Choose generators $x_\alpha \in M_{n_\alpha}$ such that the natural map
	\[	R_*\{x_\alpha\} \to M_*	\]
	becomes an isomorphism after $L$-completion. Each $x_\alpha$ corresponds to a map of spectra $S^{n_\alpha} \to M$, and they assemble to a map of $K(1)$-local $R$-modules
	\[	\bigvee \Sigma^{n_\alpha} R \to M.	\]
	This is an equivalence by a result of Hovey \cite[Theorem 7.3]{HoveySS}, which states that the functor $\pi_*$ sends ($K(1)$-local) coproducts to ($L$-complete) direct sums. The converse also follows from Hovey's result.
	\end{proof}
	
	Note that Hovey's proof uses the same, height-1-specific fact that direct sums are exact in $\Mod^\wedge_{R_*}$.
	
	\begin{prop}\label{prop: pro-free tensor}
	Suppose that $R$ is a $K(1)$-local homotopy commutative ring spectrum and $M$ and $N$ are $R$-modules, such that $M_*$ is pro-free over $R_*$. Then the natural map of $L$-complete modules,
	\[	M_* \otimes_{R_*} N_* \to \pi_*(M \sma_R N),	\]
	is an isomorphism.
	\end{prop}
	
	\begin{proof}
	By the previous lemma, we can write $M$ as a wedge of suspensions of $R$, 
	\[	M \simeq \bigvee \Sigma^{n_\alpha}R \simeq R \sma \bigvee S^{n_\alpha}	\]
	(using the fact that the $K(1)$-local smash product is a left adjoint, so distributes over the $K(1)$-local coproduct). Thus,
	\[	M \sma_R N \simeq N \sma \bigvee S^{n_\alpha} \simeq \bigvee \Sigma^{n_\alpha} N.	\]
	Using Hovey's theorem again \cite[Theorem 7.3]{HoveySS}, we obtain
	\[	\pi_*(M \sma_R N) \cong L_0\left(\bigoplus \Sigma^{n_\alpha} N_*\right) \cong L_0(F_* \otimes_{R_*} N_*),	\]
	where $F_*$ is the free graded $R_*$-module on generators in the degrees $n_\alpha$. By \cite[A.7]{HoveyStrickland},
	\[	\pi_*(M \sma_R N) \cong L_0(L_0(F_*) \otimes_{R_*} N_*) \cong M_* \otimes_{R_*} N_*.	\]
	It is clear that this isomorphism is induced by the natural map.
	\end{proof}
	
	\subsection{Homological algebra of \texorpdfstring{$L$-complete}{L-complete} Hopf algebroids}\label{subsec: Lcomplete Hopf algebroids}
	
	We now turn to the problem of homological algebra over an $L$-complete Hopf algebroid. We begin with some definitions generalizing those of \cite{Baker}.
	\begin{defn}\label{Lcomplete Hopf algebroid}
	A \textbf{$L$-complete Hopf algebroid} is a cogroupoid object $(R_*, \Gamma_*)$ in $\CAlg_*^\wedge$, such that $\Gamma_*$ is pro-free as a left $R_*$-module. As usual, we write 
	\begin{align*}
		\eta_L, \eta_R: R_* &\to \Gamma_*  &\text{ for the left and right units,} \\
		\Delta: \Gamma_* &\to \Gamma_* \otimes_{R_*} \Gamma_* &\text{ for the comultiplication,} \\
		\epsilon: \Gamma_* &\to R_* &\text{ for the counit, and} \\
		\chi: \Gamma_* &\to \Gamma_* &\text{ for the antipode.}
	\end{align*}
	\end{defn}
	
	Note that $\chi$ gives an isomorphism between $\Gamma_*$ as a left $R_*$-module and $\Gamma_*$ as a right $R_*$-module, so that $\Gamma_*$ is also pro-free as a right $R_*$-module.
	
	\begin{rmk}
	At heights higher than 1, one has to deal with the fact that the left and right units generally do not act in the same way on the generators $(p, u_1, \dotsc, u_{n-1})$ with respect to which $L$-completeness is defined. Thus, Baker's definition has the additional condition that the ideal $(p, u_1, \dotsc, u_{n-1})$ is invariant. At height 1, this condition is trivial.
	\end{rmk}
	
	\begin{defn}\label{Lcomplete comodule}
	Let $(R_*, \Gamma_*)$ be an $L$-complete Hopf algebroid. A \textbf{left comodule} over $(R_*, \Gamma_*)$ (a \textbf{left $\Gamma_*$-comodule} for short) is $M_* \in \Mod^\wedge_{R_*}$ together with a coaction map
	\[	\psi: M_* \to \Gamma_* \otimes_{R_*} M_*	\]
	such that the diagrams
	\[	\xymatrix{ M_* \ar[r]^-{\psi} \ar[d]_{\psi} & \Gamma_* \otimes_{R_*} M_* \ar[d]^{\psi \otimes 1} & M_* \ar@{=}[dr] \ar[r]^-{\psi} & \Gamma_* \otimes_{R_*} M_* \ar[d]^{\epsilon \otimes 1}  \\
		\Gamma_* \otimes_{R_*} M_* \ar[r]_-{1 \otimes \Delta} & \Gamma_* \otimes_{R_*} \Gamma_* \otimes_{R_*} M_* & & M_*}	\]
	commute. Write $\Comod^\wedge_{\Gamma_*}$ for the category of left $\Gamma_*$-comodules.
	\end{defn}
	
	\begin{lem}
	The category of left $\Gamma_*$-comodules is abelian, and the forgetful functor $\Comod^\wedge_{\Gamma_*} \to \Mod^\wedge_{\Gamma_*}$ is exact.
	\end{lem}
	
	\begin{proof}
	Suppose that
	\[	0 \to K_* \to M_* \stackrel{f}{\to} N_* \to 0	\]
	is an exact sequence of $R_*$-modules, and $f$ is a map of $\Gamma_*$-comodules. A coaction map can then be defined on $K_*$ via the diagram
	\[	\xymatrix{ 0 \ar[r] & K_* \ar@{-->}[d] \ar[r] & M_* \ar[d] \ar[r]^f & N_* \ar[d] \ar[r] & 0 \\
		0 \ar[r] & \Gamma_* \otimes_{R_*} K_* \ar[r] & \Gamma_* \otimes_{R_*} M_* \ar[r] & \Gamma_* \otimes_{R_*} N_* \ar[r] & 0. }	\]
	The bottom sequence is exact because $\Gamma_*$ is flat in $\Mod_{R_*}^\wedge$, by \Cref{lem: pro-free flat}. One checks that this structure makes $K_*$ a comodule by the usual diagram chase. A similar proof works for cokernels.
	\end{proof}
	
	\begin{defn}\label{defn: extended comodule}
	An \textbf{extended comodule} is one of the form
	\[	M_* = \Gamma_* \otimes_{R_*} N_*,	\]
	where $N_* \in \Mod^\wedge_{R_*}$, with coaction $\Delta \otimes 1_{N_*}$.
	\end{defn}
	
	When working with uncompleted Hopf algebroids, one next constructs enough injectives in the comodule category by showing that a comodule extended from an injective $R_*$-module is injective \cite[A1.2.2]{Ravenel}. One cannot do this in this case, because $\Mod^\wedge_*$ does not have enough injectives \cite[Section 1.1]{HoveySS}. For example, if $I$ is an injective $L$-complete $\Z_p$-module containing a copy of $\Z/p$, then one can inductively construct extensions $\Z/p^n \to I$ and thus a nonzero map $\Z/p^\infty \to I$ -- but this means that $I$ is not $L$-complete. Thus, one instead has to use relative homological algebra. We take the following definitions from \cite[Section 2]{BarthelHeard}.
	
	\begin{defn}\label{defn: relative injective}
	A \textbf{relative injective} comodule is a retract of an extended comodule. A \textbf{relative monomorphism} of comodules is a comodule map $M_* \to N_*$ which is a split injection as a map of $R_*$-modules. A \textbf{relative short exact sequence} is a sequence
	\[	M_* \stackrel{f}{\to} N_* \stackrel{g}{\to} P_*	\]
	where the image of $f$ is the kernel of $g$, and $f$ is a relative monomorphism. A \textbf{relative injective resolution} of a comodule $M_*$ is a sequence
	\[	M_* = J^{-1}_* \to J^0_* \to J^1_* \to \dotsb	\]
	where 
	\begin{itemize}
		\item each $J^s_*$ is relative injective for $s \ge 0$,
		\item each composition $J^{s-1}_* \to J^s_* \to J^{s+1}_*$ is zero,
		\item and if $C^s_*$ is the cokernel of $J^{s-1}_* \to J^s_*$, the sequences
		\[	C^{s-1}_* \to J^s_* \to C^s_*	\]
		are relatively short exact.
	\end{itemize}
	\end{defn}
	
	\begin{defn}\label{defn: relative ext}
	Let $M_*$ and $N_*$ be two comodules over $(R_*, \Gamma_*)$. Let $J^\bullet_*$ be a relative injective resolution of $N_*$. Define
	\[	\wh{\Ext}_{\Gamma_*}(M_*, N_*)	\]
	to be the cohomology of the complex $\Hom_{\Comod^\wedge_{\Gamma_*}}(M_*, J^\bullet_*)).$
	
	Following our general conventions, we will simply write $\Ext_{\Gamma_*}(M_*, N_*)$ for this functor, where this does not cause confusion.
	\end{defn}
	
	\begin{prop}\label{prop: relative ext facts}\mbox{}
	\begin{enumerate}[(a)]
	\item	Every comodule has a relative injective resolution.
	\item The definition of $\wh{\Ext}$ above is independent of the choice of resolution.
	\item We have
	\[	\Ext^0_{\Gamma_*}(M_*, N_*) = \Hom_{\Comod^\wedge_{\Gamma_*}}(M_*, N_*).	\]
	\item If $N_*$ is relatively injective, then $\Ext^s_{\Gamma_*}(M_*, N_*)$ vanishes for $s > 0$.
	\item If $N_* = \Gamma_* \otimes_{R_*} K_*$ for an $R_*$-module $K_*$, then $\Ext^0_{\Gamma_*}(M_*, N_*) = \Hom_{\Mod^\wedge_{R_*}}(M_*, R_*)$.
	\end{enumerate}
	\end{prop}
	
	\begin{proof}
	The first three statements follow from identical arguments to those in \cite[2.11, 2.12, 2.15]{BarthelHeard}. (One should note, in particular, that if $M_*$ is a comodule, the coaction
	\[	M_* \to \Gamma_* \otimes_{R_*} M_*	\]
	is a relative monomorphism into a relative injective.) Statement (d) is then trivial, as we can take $N_*$ to be its own relative injective resolution. For (e), we use (c) and the adjunction
	\[	\Hom_{\Comod^\wedge_{\Gamma_*}}(M_*, \Gamma_* \otimes_{R_*} K_*) \cong \Hom_{\Mod^\wedge_{R_*}}(M_*, K_*).	\]
	\end{proof}
	
	\begin{prop}\label{prop: K1local ASS}
	Let $R$ be a $K(1)$-local homotopy commutative ring spectrum such that $R_*R$ is pro-free over $R_*$.Then for any $K(1)$-local spectrum $X$, the $K(1)$-local $R$-based Adams spectral sequence for $X$ has $E_2$ page
	\[	E_2 = \Ext_{\Comod^\wedge_{R_*R}}(R_*, R_*X).	\]
	\end{prop}
	
	\begin{proof}
	This spectral sequence is the same as the Bousfield-Kan homotopy spectral sequence of the cosimplicial object
	\[	C^\bullet := R^{\sma \bullet + 1} \sma X.	\]
	This is of the form
	\[	E_1 = \pi_*(R^{\sma \bullet + 1} \sma X) \Rightarrow \pi_*\Tot(C^\bullet).	\]
	By \Cref{prop: pro-free tensor}, we have
	\[	\pi_*(R^{\sma s + 1} \sma X) = R_*R^{\otimes_{R_*} s} \otimes_{R_*} R_*X,	\]
	which is a resolution of $R_*X$ by extended comodules, so that the $E_2$ page is precisely $\Ext_{R_*R}(R_*, R_*X)$.
	\end{proof}
	
	We next discuss convergence of the spectral sequence. The Bousfield-Kan spectral sequence converges conditionally to the homotopy of its totalization, so this spectral sequence converges conditionally to $\pi_*X$ if and only if the map
	\[	X \to \holim R^{\sma \bullet + 1} \sma X	\]
	is an equivalence. Questions of this type were first studied by Bousfield \cite{Bousfield79}, and in the local case by Devinatz-Hopkins \cite{DH}. We recall their definitions here:
	
	\begin{defn}[{\cite[Appendix I]{DH}}]
	Let $R$ be a $K(1)$-local homotopy commutative ring spectrum. The class \textbf{$K(1)$-local $R$-nilpotent} spectra is the smallest class $\mathcal{C}$ of $K(1)$-local spectra such that:
	\begin{enumerate}
		\item $R \in \mathcal{C}$,
		\item $\mathcal{C}$ is closed under retracts and cofibers,
		\item and if $X \in \mathcal{C}$ and $Y$ is an arbitrary $K(1)$-local spectrum, then $X \sma Y \in \mathcal{C}$.
	\end{enumerate}
	\end{defn}
	
	\begin{prop}[{\cite[Appendix I]{DH}}]\label{prop: K1local ASS convergence}
	Assume that $X$ is $K(1)$-local $R$-nilpotent. Then the $K(1)$-local $R$-based Adams spectral sequence converges conditionally to $\pi_*X$.
	\end{prop}
	
	Finally, we write down a change of rings theorem, which is an immediate generalization of \cite[Theorem 3.3]{HoveySadofsky}:
	
	\begin{prop}\label{prop: change of rings}
	Suppose that $(A, \Gamma_A) \to (B, \Gamma_B)$ is a morphism of $L$-complete Hopf algebroids such that the natural map
	\[	B \otimes_A \Gamma_A \otimes_A B \to \Gamma_B	\]
	is an isomorphism, and such that there exists a map $B \otimes_A \Gamma \to C$ such that the composition
	\[	A \stackrel{1 \otimes \eta_R}{\to} B \otimes_A \Gamma \to C	\]
	is pro-free. Then the induced map
	\[	\Ext_{\Gamma_A}^*(A, A) \to \Ext_{\Gamma_B}^*(B, B)	\]
	is an isomorphism.
	\end{prop}

	\begin{proof}
	Virtually the same proof as in \cite{HoveySadofsky} works here. Indeed, the standard cobar complex
	\[	N(A, \Gamma_A)^\bullet = [A \to \Gamma_A \to \Gamma_A \otimes_A \Gamma_A \to \dotsb]	\]
	is a resolution of $A$ by extended comodules, and so can be used to compute Ext in exactly the same way as in the uncompleted case. Hovey and Sadofsky then give a double complex with the properties that
	\begin{enumerate}
		\item its homology in the horizontal direction is the cobar complex $N(B, \Gamma_B)^\bullet$ computing $\Ext_{\Gamma_B}(B, B)$,
		\item and, writing $R^\bullet$ for the homology in the vertical direction, there is a map of complexes
		\[	g:N(A, \Gamma_A)^\bullet \to R^\bullet	\]
		which becomes an isomoprhism after tensoring both sides on the left with $B \otimes_A \Gamma_A$.
	\end{enumerate}
	In particular, $C \otimes_A g$ is an isomorphism of complexes. But $C$ is pro-free over $A$, so $g$ is an isomorphism by \Cref{lem: pro-free flat}. The remainder of the argument follows formally from a consideration of the two spectral sequences associated to the double complex.
	\end{proof}

\subsection{The Hopf algebroids for \texorpdfstring{$K$}{K} and \texorpdfstring{$KO$}{KO}}\label{subsec: Hopf algebroids K KO}

The $K$-theory spectrum has a group action by $\Z_p^\times$ via $E_\infty$ ring maps. For $k \in \Z_p^\times$, we write $\psi^k$ for the corresponding endomorphism of $K$, called the $k$th \textbf{Adams operation}. On homotopy, writing $u$ for the Bott element, we have
\begin{equation}\label{eq: Adams ops on K}
	\psi^k: K_*\to K_*: \quad u^n\mapsto k^n\, u^n.
\end{equation}

The group $\Z_p^\times$ has a maximal finite subgroup $\mu$ of order $p-1$, and we write $KO = K^{h\mu}$. (This agrees with the $p$-completion of the real $K$-theory spectrum at $p=2$ and 3). Then $KO$ inherits an action by the topologically cyclic group $\Z_p^\times/\mu$, which we also refer to as an action by Adams operations.

The Adams operations give us a way to analyze the completed cooperations algebras $K_*K$ and $KO_*KO$. Define
\[	\Phi_K: K_*K \to  \Cts(\Z_p^\times, K_*)	\]
as adjoint to the map
\[	K_*K \times \Z_p^\times \to K_*	\]
defined by taking an element $x: S^0\to K\wedge K$ and $p$-adic unit $k\in \Z_p^\times$ to the composite
	\[
	\begin{tikzcd}
		S^0\arrow[r,"x"] & K\wedge K\arrow[r,"K\wedge \psi^k"] & K\wedge K\arrow[r,"m"] & K
	\end{tikzcd}
	\]
Likewise, there is a map
\[	\Phi_{KO}: KO_*KO \to \Cts(\Z_p^\times/\mu, KO_*).	\]

\begin{thm}[{cf. \cite{hovey2004}}]\label{thm: K Hopf algebroid}
	The map 
	\[	\Phi_K:K_*K \to \Cts(\Z_p, K_*)	\]
	is an isomorphism. It induces an isomorphism of Hopf algebras
	\[	(K_*, K_*K) \cong (K_*, \Cts(\Z_p^\times, K_*)),	\]
	where the latter has the following Hopf algebra structure:
	\begin{itemize}
		\item The unit $\eta_L = \eta_R: K_*\to \Cts(\Z_p^\times, K_*)$ is the inclusion of constant functions.
		\item The coproduct,
		\[
			\Delta:  \Cts(\Z_p^\times, K_*)\to \Cts(\Z_p^\times, K_*)\otimes \Cts(\Z_p^\times, K_*)\cong \Cts(\Z_p^\times\times \Z_p^\times, K_*),
		\]
		is given by sending a function $f$ to the function $(a,b)\mapsto f(ab)$.
		\item The antipode $\Cts(\Z_p^\times, K_*)\to \Cts(\Z_p^\times, K_*)$ sends a function $f$ to $a\mapsto f(a^{-1})$. 
		\item The augmentation map $\Cts(\Z_p^\times, K_*)\to K_*$ is given by evaluation at $1$.
	\end{itemize}
	Analogous statements hold for $KO$.
\end{thm}

\begin{rmk}
The cooperations algebra $K_*K$ carries \textit{two} actions by Adams operations, coming from the two copies of $K$. Given $f \in K_0K$, we can represent $f$ both as a map $f:S^0 \to K \sma K$ and as an element of $\Cts(\Z_p^\times, K_0)$. Then, for $a,b\in \Z_p^\times$, we have 
\begin{equation}\label{eq:leftAdamsop}
		((\psi^a\wedge K)\circ f)(b) = f(ab)
	\end{equation}
	and
	\begin{equation}\label{eq:rightAdamsop}
	((K\wedge \psi^a)\circ f)(b) = \psi^a(f(a^{-1}b)).
	\end{equation}
\end{rmk}

Now suppose that $M_*$ is an $L$-complete $K_*K$-comodule with coaction $\psi_{M_*}$. Then there is a map
\begin{align*}
	M_* &\stackrel{\psi_{M_*}}{\to} K_*K\otimes_{K_*} M_* \\
	&\cong \Cts(\Z_p^\times, K_*)\otimes_{K_*}M_* \\
	&\cong \Hom_{\Mod^\wedge_*}(\Z_p[[\Z_p^\times]], K_*) \otimes_{K_*} M_* \\
	&\to \Hom_{\Mod^\wedge_*}(\Z_p[[\Z_p^\times]], M_*).
\end{align*}
Here $\Hom_{\Mod^\wedge_*}$ is the ordinary space of maps between $\Z_p$-modules, which is automatically $L$-complete when the modules are $L$-complete \cite[section 1.1]{HoveySS}. As $\Mod^\wedge_*$ is closed symmetric monoidal, this map is adjoint to one of the form
\begin{equation}\label{eq: Adams ops comodule}
	M_* \otimes \Z_p[[\Z_p^\times]] \to M_*.
\end{equation}
In the case where $M_*$ is $p$-complete, this defines a continuous group action by $\Z_p^\times$ on $M_*$. If $M_*$ is merely $L$-complete, then one still gets a group action by $\Z_p^\times$ on $M_*$, and the only reasonable definition of ``continuous group action'' appears to be that it extends to a map of $L$-complete modules of the form \eqref{eq: Adams ops comodule}. In either case, we call this the action by Adams operations on $M_*$. Of course, if $M_*$ is the completed $K$-theory of a spectrum $X$, $M_* = \pi_*L_{K(1)}(K \sma X)$, then this action is induced by the Adams operations on $K$.

If $M_*$ is $p$-complete then the standard relative injective resolution of $M_*$,
\[	M_* \to K_*K \otimes_{K_*} M_* \to K_*K \otimes_{K_*} K_*K \otimes_{K_*} M_* \to \dotsb,	\]
is isomorphic to the complex of continuous $\Z_p^\times$-cochains,
\[	M_* \to \Cts(\Z_p^\times, M_*) \to \Cts(\Z_p^\times \times \Z_p^\times, M_*).	\]
Thus, we can identify the relative Ext of \Cref{defn: relative ext} with continuous group cohomology:
\[	\Ext_{K_*K}(K_*, M_*) = H^*_{cts}(\Z_p^\times, M_*).	\]

Similar remarks apply to $KO$: a $KO_*KO$-comodule $M_*$ has a continuous group action by $\Z_p^\times/\mu$, and if $M_*$ is $p$-complete, we have
\[	\Ext_{KO_*KO}(KO_*, M_*) = H^*_{cts}(\Z_p^\times/\mu, M_*).	\]
(Again, if $M_*$ is merely $L$-complete, then one should instead take these Ext groups as a definition of continuous group cohomology with coefficients in $M_*$!)

One recovers the familiar $K(1)$-local Adams spectral sequences based on $KO$ as an immediate consequence. 
\begin{prop}\label{prop: K1local ASS KO}
Let $X$ be a $K(1)$-local spectrum. Then there is a strongly convergent Adams spectral sequence
\begin{align*}
	E_2 = \Ext_{KO_*KO}^{s,t}(KO_*, KO_*X) = H^s_{cts}(\Z_p^\times/\mu, KO_tX) &\Rightarrow \pi_{t-s}X.
\end{align*}
\end{prop}

\begin{proof}
The calculation of the $E_2$ pages follows from the above discussion and \Cref{prop: K1local ASS}. Since $\Z_p^\times/\mu$ has cohomological dimension 1, the spectral sequence collapses at $E_2$, and in particular, converges strongly. To establish that the limit is $\pi_*X$, we must show that every $K(1)$-local $X$ is $K(1)$-local $KO$-nilpotent (see \Cref{prop: K1local ASS convergence}). But the sphere is a fiber of copies of $KO$, so $S$ is $K(1)$-local $KO$-nilpotent, so the same is true for arbitrary $X$.
\end{proof}

\section{Cones on \texorpdfstring{$\zeta$}{zeta}}\label{sec: Cone zeta}

In this section, we will analyze the cone on $\zeta$ as well as the $E_\infty$-cone $T_\zeta$ on $\zeta$. These spectra play an important role in the construction and calculation of the homotopy groups of $\tmf$, and will be important for later parts of this paper. Most of the material in this section can be found in \cite{K1localrings}.

\subsection{The spectrum cone on \texorpdfstring{$\zeta$}{zeta}}

The fiber sequence
\[	
\begin{tikzcd}
	S\arrow[r] & KO \arrow[r,"\psi^g-1"] & KO
\end{tikzcd}
\]
gives a long exact sequence on homotopy groups
\[
\begin{tikzcd}
	\cdots \arrow[r] & \pi_n S\arrow[r] & \pi_n KO \arrow[r, "\psi^g-1"] & \pi_n KO\arrow[r, "\partial"] & \pi_{n-1} S\arrow[r] & \cdots .
\end{tikzcd}
\]
Recall that the action of $\psi^g$ on $\pi_0KO$ is trivial, so the connecting homomorphism gives an isomorphism 
\[
\Z_p=\pi_0KO \cong \pi_{-1}S.
\]
This isomorphism does depend on the choice of topological generator $g$. We let $\zeta:= \partial(1)$, and we define $C(\zeta)$ to be the cone on $\zeta$, i.e. the cofibre
\[
\begin{tikzcd}
	S^{-1}\arrow[r,"\zeta"] & S\arrow[r] & C(\zeta).
\end{tikzcd}
\]
Since $\pi_{-1}KO=0$, we get a morphism of cofibre sequences
\begin{equation}\label{eq: morphism-cofibre-sequences}
\begin{tikzcd}
	S^{-1} \arrow[r,"\zeta"] & S^0\arrow[d,"="] \arrow[r]& C(\zeta)\arrow[d,"\iota"]\arrow[r,"\delta"] & S^0\arrow[d,"\eta"]\arrow[r, "\zeta"] & S^1\arrow[d,"="] \\
	&S^0\arrow[r,"\eta"] & KO\arrow[r,"\psi^g-1"] & KO \arrow[r]& S^1
\end{tikzcd}.
\end{equation}
The morphism $\iota$ is a nullhomotopy of $\eta\circ \zeta$.

Since $\zeta$ is nullhomotopic in $KO$, the top cofibre sequence in \eqref{eq: morphism-cofibre-sequences} splits after smashing with $KO$, giving $KO \sma C(\zeta) \simeq KO \sma (S^0 \vee S^0)$. In fact, there is a canonical splitting, coming from the diagram
\begin{equation}\label{eq:diagramCzeta}
\begin{tikzcd}
	KO\wedge S^0\arrow[rr] \arrow[d,"="]& & KO\wedge C(\zeta)\arrow[rr,"KO\wedge \delta"]\arrow[d, "KO\wedge \iota"] && KO\wedge S^0\arrow[d, "KO\wedge \eta"]\\\
	KO\wedge S^0\arrow[rr, "KO\wedge \eta"] \arrow[drr, "="]&& KO\wedge KO\arrow[d,"m"] \arrow[rr, "KO\wedge (\psi^g-1)"] && KO\wedge KO \\
	& & KO &&
\end{tikzcd}
\end{equation}
We see that
\[	m \circ (KO \sma \iota): KO \sma C(\zeta) \to KO	\]
splits the inclusion $KO \to KO \sma C(\zeta)$. Thus, we can choose classes $a, b \in KO_0C(\zeta)$ by
\begin{align*}
m(KO \sma \iota)(a) = 1, & & (KO \sma \delta)(a) = 0, \\
m(KO \sma \iota)(b) = 0, & & (KO \sma \delta)(b) = -1,
\end{align*}
and $\{a, b\}$ is a $KO_*$-module basis for $KO_*C(\zeta)$.

\begin{prop}\label{prop: b}
	Under the morphism 
	\[	KO\wedge\iota: KO_0C(\zeta)\to KO_0KO = \Cts(\Z_p^\times/\mu, \Z_p),	\]
	the element $a$ is mapped to the constant function 1 and $b$ is mapped to the unique group homomorphism sending $g$ to 1.
\end{prop}
\begin{proof}
We use the formulas from \Cref{thm: K Hopf algebroid} and \eqref{eq:rightAdamsop}. For the sake of brevity, let $\ol{a}$ and $\ol{b}$ be the images of $a$ and $b$ under $KO \sma \iota$, which we think of as continuous functions from the topologically cyclic group $\Z_p^\times/\mu$ to $\Z_p$. Since $m(\ol{a}) = 1$, by \Cref{thm: K Hopf algebroid}, the function $\ol{a}$ satisfies $\ol{a}(1) = 1$.
We also have
\[	(KO \sma (\psi^g - 1))(\ol{a}) = (KO \sma \eta)(KO \sma \delta)(a) = 0,	\]
and by \eqref{eq:rightAdamsop}, together with the fact that $\psi^g$ acts trivially on $KO_0$,
\[	\ol{a}(g^{-1}n) - \ol{a}(n) = 0	\]
for any $n \in \Z_p^\times/\mu$. Together with continuity of $\ol{a}$, this implies that $\ol{a}$ is constant.

Applying the same arguments to $\ol{b}$, we obtain
\[	\ol{b}(1) = 0, \quad \ol{b}(g^{-1}n) - \ol{b}(n) = -1.	\]
It follows that
\[	\ol{b}(g^k) = k	\]
for any $k \in \Z$, and by continuity, for any $k \in \Z_p$.
\end{proof}

\begin{cor}
The map
\[	\iota_*: KO_0C(\zeta) \to KO_0KO	\]
is injective.
\end{cor}

\begin{proof}
One just has to observe that the functions $\ol{a}$, $\ol{b}$ are linearly independent in $KO_0KO$.
\end{proof}

\begin{cor}\label{cor: b Adams}
	In $KO_*C(\zeta)$, the Adams operations fix $a$ and $\psi^g(b) = b+a$.
\end{cor}

\begin{proof}
By the previous corollary, the Adams operations can be calculated in $KO_0KO$, where they are given by \eqref{eq:leftAdamsop}.
\end{proof}

\begin{cor}
We have
\[	K_*C(\zeta) \cong K_*\{a, b\},	\]
where the Adams operations fix $a$ and satisfy
\[	\psi^g(b) = b + a, \quad \psi^{\omega}(b) = b.	\]
\end{cor}

\begin{proof}
The $KO$-module $KO \sma C(\zeta)$ is free on the generators $\{a,b\}$, so $K \sma C(\zeta)$ is free on the same generators as a $K$-module. Since the generators of $K_*C(\zeta)$ are in the image of $KO_*C(\zeta)$, they are fixed by $\psi^{\omega}$.
\end{proof}

\subsection{The \texorpdfstring{$E_\infty$}{E infinity}-cone on \texorpdfstring{$\zeta$}{zeta}}

The previous subsection allows us to start the analysis of the $E_\infty$-cone on $\zeta$. 

\begin{defn}
	The spectrum $T_\zeta$ is defined by the following homotopy pushout square in the category $\CAlg$. 
	\begin{equation}\label{eq: pushout Tzeta}
	\xymatrix{ \PP(S^{-1}) \ar[r]^-0 \ar[d]_\zeta & S^0 \ar[d] \\ S^{0} \ar[r] & \pushoutcorner T_\zeta }	
	\end{equation}
\end{defn}

Just as $C(\zeta)$ classifies nullhomotopies of $\zeta$ in spectra equipped with a map from $S^0$, $T_\zeta$ classifies nullhomotopies of $\zeta$ in $E_\infty$-algebras. That is, there is a natural equivalence of mapping spaces
\[	 \CAlg(T_\zeta,R) \simeq \Sp_{S^0/}(C(\zeta),R).	\]
In particular, there is a canonical morphism $C(\zeta)\to T_\zeta$, and a canonical factorization
\begin{equation}\label{eq: factorization iota}
\begin{tikzcd}
	C(\zeta)\arrow[r] \arrow[rr, bend right, swap,"\iota"]& T_\zeta\arrow[r] & KO.
\end{tikzcd}
\end{equation}
We also have the following. 

\begin{prop}\label{prop: homology of cone on zeta}
Let $R$ be any $E_\infty$-algebra such that $\pi_{-1}R=0$. Then there is an equivalence in $\CAlg_{R}$:
\[
R\wedge T_\zeta\simeq R\wedge \PP(S^0).
\]
\end{prop}
\begin{proof}
Smashing $R$ with the pushout diagram for $T_\zeta$ produces a pushout diagram
\[
\begin{tikzcd}
R\wedge \PP(S^{-1})\arrow[r,"*"]\arrow[d,"R\wedge\zeta"] & R\wedge S^0\arrow[d]\\
R\wedge S^0\arrow[r] & R\wedge T_\zeta
\end{tikzcd}
\]
Observe the equivalence $\PP_R(R\wedge S^0)\simeq R\wedge \PP(S^0)$. Note that $R\wedge \zeta$ is adjoint to the map 
\[
R\wedge \zeta: R\wedge S^{-1}\to R\wedge S^0
\]
in $R$-modules. This morphism is itself adjoint to the map 
\[
\zeta: S^{-1}\to S^0\to R\wedge S^0
\]
in $S$-modules. As $\pi_{-1}R=0$, this map is null, which implies $R\wedge \zeta$ is null in $R$-modules. Thus the morphism $R\wedge\zeta$ in $\CAlg_{R}$ is adjoint to the null morphism. So the pushout diagram is in fact the pushout of the following
\[
\begin{tikzcd}
\PP_R(R\wedge S^{-1})\arrow[r,"0"] \arrow[d,"0"]& R\\
R &
\end{tikzcd}
\]
which gives $R\wedge \PP(S^0)$. 
\end{proof}

\begin{cor}
	There is an equivalence of $KO$-algebras
	\[
	KO\wedge T_\zeta\simeq KO\wedge \PP(S^0).
	\]
\end{cor}

More explicitly, we can choose this equivalence so that the following diagram commutes:
\[
\begin{tikzcd}\label{eq: cone and Einfty cone}
KO \wedge C(\zeta) \arrow[r, "a \vee b"] \arrow[d] & KO \wedge (S^0 \vee S^0) \arrow[d] \\
KO \wedge T_\zeta \arrow[r] & KO \wedge \PP(S^0)
\end{tikzcd}
\]
Here, the map
\[	S^0 \vee S^0 \to \PP(S^0)	\]
is the unit on the left summand, and the inclusion of the generator on the right one. This allows us to calculate the $KO$-homology of $T_\zeta$ completely.

\begin{cor}
As a $\theta$-algebra over $KO_*$,
\[	KO_*T_\zeta \cong KO_* \otimes \TT(b), \text{ with }\psi^g(b) = b+1,	\]
where $b$ is the image of the element of $KO_0C(\zeta)$ described in \Cref{prop: b}. Likewise,
\[	K_*T_\zeta \cong K_* \otimes \TT(b), \text{ with }\psi^g(b) = b+1, \,\, \psi^{\omega}(b) = b.	\]
\end{cor} 
\begin{proof}
	This is a consequence of the $E_\infty$ equivalence $KO \sma T_\zeta \simeq KO \sma \PP(S^0)$, McClure's theorem \ref{thm: McClure}, and the commutativity of \eqref{eq: cone and Einfty cone}. Since $b$ is in the image of $KO_0C(\zeta)$, its Adams operations follow from \Cref{cor: b Adams}. As the Adams operations on $KO_*$ are known and $\psi^g$ commutes with $\theta$, the calculation of $\psi^g(b)$ determines the Adams operations on all of $KO_*T_\zeta = KO_* \otimes \TT(b)$. Tensoring up to $K$, one also gets the formula for $K_*T_\zeta$.
\end{proof}

\subsection{The homotopy groups of \texorpdfstring{$T_\zeta$}{Tzeta}}

In this subsection we compute the homotopy groups of $T_\zeta$. This has been done before in \cite{K1localrings} and \cite{Laures2004}. As this calculation is important for the work on co-operations to follow, we review it here in detail.

We may approach the homotopy groups of $T_\zeta$ using the $KO$-based Adams spectral sequence, which we saw in \Cref{prop: K1local ASS KO} takes the form
\begin{equation}\label{eq: HFPSS for Tzeta}
E_2 = \Ext_{KO_*KO}(KO_*, KO_*T_\zeta) = H^*_{cts}(\Z_p^\times/\mu; KO_*T_\zeta)\implies \pi_*T_\zeta.
\end{equation}
The key point of Hopkins' calculation in \cite{K1localrings} is as follows:
\begin{thm}[{\cite{K1localrings}, \cite{Laures2004}}]\label{thm: KO homology of Tzeta}
	The $KO$-homology of $T_\zeta$ is an extended $KO_*KO$-comodule. More specifically, there is an isomorphism of $KO_*KO$-comodules
	\[
	KO_*T_\zeta \cong KO_*KO \otimes \TT(f) \cong \Cts(\Z_p^\times/\mu, \Z_p) \otimes KO_* \otimes \TT(f),
	\]
	where $f = \psi^p(b)-b$, and $\TT(f)$ has trivial coaction.
\end{thm}

This allows an immediate derivation of $\pi_*T_\zeta$.

\begin{cor}
	The homotopy groups of $T_\zeta$ are
	\[	
	\pi_*T_\zeta \cong KO_* \otimes \TT(f).
	\]
\end{cor}

\begin{proof}
By \Cref{prop: relative ext facts}, the cohomology of an extended comodule is concentrated in degree zero, and
\[	\Ext^0_{KO_*KO}(KO_*, KO_*KO \otimes \TT(f)) = \Hom_{KO_*}(KO_*, KO_* \otimes \TT(f)) = KO_* \otimes \TT(f).	\]
\end{proof}

The proof of Theorem \ref{thm: KO homology of Tzeta} will take up the remainder of this section. As it is somewhat involved, let us give an outline first. The map $T_\zeta \to KO$ induces a map
\[	KO_0T_\zeta \to KO_0KO = \Cts(\Z_p^\times/\mu,\Z_p).	\]
This is a map of $\theta$-algebras and of $KO_0KO$-comodules, and there are also natural Hopf algebra structures on both objects making it a Hopf algebra map. We also consider the leaky $\lambda$-ring structures of \Cref{lambda-ring functors}. Using all this structure, we prove that the Hopf algebra kernel is just $\TT(f)$, and construct a coalgebra splitting. This implies that $KO_0T_\zeta$ is an induced $KO_0KO$-comodule by a general theorem about Hopf algebras. Finally, one can explicitly construct $\Z_p^\times/\mu$-invariant elements in nonzero degrees of $KO_*T_\zeta$, multiplication by which allows us to transport the result in degree zero to nonzero degrees.

\begin{lem}\label{i injective flat}
The map of $\theta$-algebras $i:\TT(f) \to \TT(b)$ sending $f$ to $\psi^p(b) - b$ is injective and pro-free.
\end{lem}

\begin{proof}
Let $b_0 = b$ and $b_i = \theta_i(b)$, and likewise with $f_i$, where the operations $\theta_i$ are as defined in \Cref{thm: free theta-algebra}. Then
\[	\TT(b) = \Z_p[b_0,b_1,\dotsc] \text{ and }\TT(f) = \Z_p[f_0,f_1,\dotsc].	\]
We claim that 
\begin{equation}\label{f congruence}
	f_i \equiv b_i^p - b_i \mod{(p,b_0, \dotsc, b_{i-1})}.
\end{equation}
This is true for $i = 0$. Suppose it has been proven for $i = 0, \dotsc, n-1$. Then
\[	\psi^{p^n}(f) = \psi^{p^{n+1}}(b) - \psi^{p^n}(b),	\]
or in other words,
\begin{equation}\label{fs and bs}
f_0^{p^n} + p f_1^{p^{n-1}} + \dotsb + p^n f_n = b_0^{p^{n+1}} - b_0^{p^n} + p(b_1^{p^n} - b_1^{p^{n-1}}) + \dotsb + p^n(b_n^p - b_n) + p^{n+1} b_{n+1}.
\end{equation}
Since $f_i \equiv b_i^p - b_i$ mod $(p,b_0, \dotsc, b_{i-1})$, we have
\[	f_i \equiv 0 \mod{(p, b_0, \dotsc, b_i)},	\]
and thus
\[	p^if_i^{p^{n-i}} \equiv 0 \mod{(p^{i+n-i+1}, b_0, \dotsc, b_i)}.	\]
Thus, \eqref{fs and bs} reduces mod $(p^{n+1}, b_0, \dotsc, b_{n-1})$ to
\[	p^n f_n \equiv p^n(b_n^p - b_n) \mod{(p^{n+1}, b_0, \dotsc, b_{n-1})}	\]
or just
\[	f_n \equiv b_n^p - b_n \mod{(p, b_0, \dotsc, b_{n-1})},	\]
which is \eqref{f congruence} for $i = n$.

Thus, $\TT(b)/p = \F_p[b_0, b_1, \dotsc]$ is freely generated over $\TT(f)/p = \F_p[f_0, f_1, \dotsc]$ by the monomials $b_0^{n_0}b_1^{n_1}\dotsm$ in which all $n_i < p$ and all but finitely many of the $n_i$ are zero. By \Cref{lem: pro-free mod p}, $\TT(b)$ is pro-free over $\TT(f)$. In particular, the unit map is an injection.
\end{proof}

It will be helpful to make the identification
\[	KO_0KO = \Cts(\Z_p^\times/\mu, \Z_p) \cong \Cts(\Z_p,\Z_p)	\]
using the continuous group isomorphism
\[	\Z_p^\times \stackrel{\cong}{\to} \Z_p, \quad g \mapsto 1.	\]
By \cref{prop: b}, $b \in KO_0KO$ goes to the identity under this identification.

The map $T_\zeta \to KO$ induces a map
\[	\pi:\TT(b) = KO_0T_\zeta \to KO_0KO \cong \Cts(\Z_p, \Z_p),	\]
This is a $\theta$-algebra map, determined by the fact that $\pi(b) = \mathrm{id}$. By \Cref{K theta algebra}, $\psi^p(\pi(b)) = \pi(b)$. Thus, there is an induced map
\[	\overline{\pi}:\TT(b) \otimes_{\TT(f)} \Z_p \to \Cts(\Z_p,\Z_p)	\]
where $\TT(f) \to \Z_p$ sends all $\theta^k(f)$ to 0.

\begin{defn}
We give $\TT(b)$ and $\Cts(\Z_p,\Z_p)$ the leaky $\lambda$-ring structures $\mathcal{L}(\TT(b))$, $\mathcal{L}(\Cts(\Z_p,\Z_p))$ of \Cref{lambda-ring functors}. In each of these $\lambda$-rings, the Adams operations $\psi^k$ associated to the $\lambda$-ring structure are the identity for $k$ prime to $p$, while $\psi^p$ is equal to the operation $\psi^p$ associated to the $\theta$-algebra structure.
\end{defn}

By \Cref{Zp lambda-ring}, the $\lambda$-operations on $\phi \in \Cts(\Z_p,\Z_p)$ are given by
\[	\lambda^n(\phi)(x) = \binom{\phi(x)}{n}.	\]

\begin{lem}
The map
\[	\pi:\TT(b) \to \Cts(\Z_p,\Z_p)	\]
is a map of $\lambda$-rings.
\end{lem}

\begin{proof}
This map is obtained by applying the functor $\mathcal{L}$ to a map of $\psi$-$\theta$-algebras.
\end{proof}

\begin{prop}\label{j iso}
The map
\[	\overline{\pi}:\TT(b) \otimes_{\TT(f)} \Z_p \to \Cts(\Z_p,\Z_p)	\]
is an isomorphism.
\end{prop}

\begin{proof}
First, let's show the map is surjective. Since the map $\TT(b) \to \Cts(\Z_p,\Z_p)$ is a map of $\lambda$-rings with $\mathrm{id}_{\Z_p}$ in its image, $\lambda^k(\mathrm{id})$ is also in its image for all $k \in \N$. Observe that $\lambda^k(\mathrm{id})$ is precisely the binomial coefficient function $\beta_k:x \mapsto \binom{x}{k}$. It is a theorem of Mahler \cite[4.2.4]{Robert2000} that $\Cts(\Z_p,\Z_p)$ is generated (as a complete $\Z_p$-module) by the binomial functions $\beta_k$ for $k \in \N$. Thus the map $\pi:\TT(b) \to \Cts(\Z_p, \Z_p)$ is surjective.

We now introduce an alternative description of $\Cts(\Z_p,\Z_p)$. Any element of $\Z_p$ has a unique description
\[	a = \sum_{i\ge 0} a_ip^i	\]
where each $a_i$ is a Teichm\"uller lift, i.e., either zero or a $(p-1)$th root of unity. Define
\[	\alpha_i(a) = a_i.	\]
A continuous map $\Z_p \to \Z/p^n$ can be described in terms of a finite number of the $\alpha_i$, so we have
\[	\Cts(\Z_p,\Z/p^n) = \Z/p^n[\alpha_0,\alpha_1,\dotsc]/(\alpha_i^p - \alpha_i).	\]
Taking the limit gives
\[	\Cts(\Z_p,\Z_p) = \Z_p[\alpha_0,\alpha_1,\dotsc]/(\alpha_i^p - \alpha_i).	\]

Let $b_n := \theta_n b$ in $\TT(b)$, so $\TT(b) = \Z_p[b_0, b_1, \ldots]$. Recall the identities 
\begin{equation}\label{b psi identity}
\psi^{p^n}b = b_0^{p^n} + pb_1^{p^{n-1}}+\cdots + p^nb_n.
\end{equation}
We claim that 
\[
\pi(b_n) \equiv \alpha_n \mod p
\]
for all $n$. We proceed by induction: first, $\pi(b_0) = \mathrm{id}$ is congruent to $\alpha_0$ mod $p$. Suppose we have shown that 
\[
\pi(b_i)\equiv \alpha_i\mod p
\]
for each $i<n$. It follows that 
\[
\pi(b_i^{p^{n-i}})\equiv \alpha_i^{p^{n-i}} = \alpha_i \mod p^{n-i+1}
\]
and so 
\[
\pi(p^{i}b_i^{p^i})\equiv p^{i}\alpha_i\mod p^{n+1}. 
\]
Thus, applying $\pi$ to \eqref{b psi identity} and using the fact that $\psi^p$ is the identity on $\Cts(\Z_p,\Z_p)$, we get
\[	\mathrm{id} \equiv \alpha_0 + p\alpha_1 + \dotsb + p^{n-1}\alpha_{n-1} + p^n\pi(b_n)\mod{p^{n+1}}.	\]
But of course $\mathrm{id} = \sum p^i\alpha_i$ on the nose, so solving for $\pi(b_n)$ gives
\[	\pi(b_n) \equiv \alpha_n\mod{p}.	\]

We can now compute the kernel of $\pi$. First note that it contains each
\[	\theta_n(f) = \psi^p(b_n) - b_n.	\]
This is just because it's a $\theta$-algebra map whose kernel contains $f$, and was needed to define the map $\ol{\pi}$ in the first place. We want to show that the $\theta_n(f)$ generate the kernel of $\pi$. But we know that
\[
\pi/p: \F_p[b_0,b_1, \ldots]\to \Cts(\Z_p, \F_p) = \F_p[\alpha_0, \alpha_1, \ldots]/(\alpha_i^p-\alpha_i)
\]
sends $b_i$ to $\alpha_i$, so that $\ker(\pi/p)$ is generated by the elements $b_n^p - b_n \equiv \psi^p(b_n) - b_n$ (mod $p$).

Since $\Cts(\Z_p, \Z_p)$ is a free complete $\Z_p$-module, we have that 
\[	\Tor^1_{\Z_p}(\Cts(\Z_p,\Z_p), \F_p)=0,	\]
and so 
\[	\ker(\pi/p)\cong \ker(\pi)\otimes \F_p.	\]
Since $\TT(b)$ is $p$-adically complete and torsion free, it follows that the elements $\psi^p(b_n) - b_n$ also generate $\ker(\pi)$, concluding the proof. 
\end{proof}

\begin{lem}\label{lem: maps as hopf}
The map $i:\TT(f) \to \TT(b)$ is a map of Hopf algebras, where $\TT(f)$ and $\TT(b)$ both have the Hopf algebra structure of \Cref{ex: Hopf algebra free theta algebra single generator}. The induced Hopf algebra structure on $\Cts(\Z_p,\Z_p) = \TT(b)\sslash \TT(f)$ is the same as that induced by addition on the source $\Z_p$.
\end{lem}
\begin{proof}
The first statement follows from the fact that the functor $\TT$ naturally takes values in Hopf algebras. In particular, the diagonal map $\Delta: \TT(M)\to \TT(M)\otimes \TT(M)$ is functorial in $M$. Thus $i$ is a map of Hopf algebras. 

For the second statement, it suffices to show that the given map $\pi:\TT(b) \to \Cts(\Z_p,\Z_p)$ is a Hopf algebra map. This can be checked after tensoring with $\Q$, in which case it suffices to check that $\pi(\psi^{p^n}(b))$ is still primitive. However, we have seen that each $\psi^{p^n}(b)$ goes to the identity of $\Z_p$, which is primitive in $\Cts(\Z_p,\Z_p)$.
\end{proof}

\begin{lem}\label{s}
The map $\pi:\TT(b) \to \Cts(\Z_p, \Z_p)$ admits a coalgebra section $s:\Cts(\Z_p,\Z_p) \to \TT(b)$.
\end{lem}

\begin{proof}
By Mahler's theorem cited above, $\Cts(\Z_p,\Z_p)$ is a free complete $\Z_p$-module on the binomial functions $\beta_k:x \mapsto \binom{x}{k}$, for $k \in \N$. In the proof of \Cref{j iso}, we saw that, in terms of the $\lambda$-ring structure on $\TT(b)$, $\pi(\lambda^k(b)) = \beta_k$. We can define a continuous $\Z_p$-module section by
\[	s\left( \beta_k\right) = \lambda^k(b).	\]

It remains to see that this is also a coalgebra section. It follows from \Cref{lem: coproduct lambda algebras} that the coproduct $\Delta$ is a morphism of $\lambda$-algebras. 

The binomial functions have comultiplication
\[	\Delta(\beta_n) = \sum_{i=0}^n \beta_i \otimes \beta_{n-i}.	\]
Therefore,
\[	(s \otimes s)\Delta(\beta_n) = \sum_{i=0}^n \lambda^i(b) \otimes \lambda^{n-i}(b) = \Delta s(\beta_n).	\]
So $s$ is a coalgebra map.
\end{proof}

Equipped with the above lemmas, we can finally prove \Cref{thm: KO homology of Tzeta}. We begin by proving the degree zero part.

\begin{prop}\label{T(b) induced} 
There is an isomorphism of $\TT(f)$-modules and $KO_0KO$-comodules
\[	\TT(f) \otimes KO_0KO \cong \TT(b).	\]
\end{prop}

\begin{proof}
\textbf{Note:} \textit{For the duration of this proof, we will make all completions explicit.}

We wish to show that
\[	KO_*T_\zeta = KO_* \otimes \TT(f) \cong KO_*KO \otimes \TT(b).	\]

At this point, we have maps of complete Hopf algebras
\begin{equation}\label{Hopf algebra exact sequence}
	\TT(f) \stackrel{i}{\to} \TT(b) \stackrel{\pi}{\to} KO_0KO,
\end{equation}
such that $KO_0KO = \TT(b) \ol{\otimes}_{\TT(f)} \Z_p$, together with a coalgebra section $s$ of $\pi$. We claim that
\[	
\begin{tikzcd}
	\widehat{\phi}: \TT(f)\,\ol{\otimes}\, KO_0KO\arrow[r,"i\otimes s"] & \TT(b)\,\ol{\otimes}\, \TT(b)\arrow[r,"m"] & \TT(b) 
\end{tikzcd}
\]
is the desired isomorphism. This uses a variant of the arguments in \cite[Section 1]{MilnorMoore}. The situation is slightly complicated by the omnipresence of completion, as well as the fact that the objects involved are not graded in any manageable way.

First, we handle the completions. Let $A$ be the uncompleted polynomial ring 
\[
A:=\Z_p[f,\theta(f), \theta_2(f), \dotsc],
\] 
and likewise let $B$ be the uncompleted polynomial ring on the $\theta_n(b)$. Let $C$ be the sub-$\Z_p$-algebra of $\Cts(\Z_p,\Z_p)$ consisting of those functions which can be written as polynomials with $\Q_p$ coefficients. As an uncompleted $\Z_p$-module, $C$ is free on the $\beta_n$. The sequence \eqref{Hopf algebra exact sequence} restricts to a sequence of maps of Hopf algebras
\begin{equation}\label{Hopf algebra uncompleted exact sequence}
	A \stackrel{i}{\to} B \stackrel{\pi}{\to} C,
\end{equation}
such that $C = B \otimes_A \Z_p$, together with a coalgebra section $s$ of $\pi$. Write $\phi$ for the map
\[	A \otimes C \stackrel{i \otimes s}{\to} B \otimes B \stackrel{mult}{\to} B.	\]
Since $\widehat{\phi}$ is the completion of $\phi$, it suffices to prove that $\phi$ is an isomorphism of $A$-modules and $C$-comodules.

Now, $\phi$ is clearly an $A$-module, and it is also a map of $C$-comodules since $s$ is a map of coalgebras. We will show that $\phi$ is injective by the method of \cite[Proposition 1.7]{MilnorMoore}. Note that $C$ has a coalgebra grading in which the degree of $\beta_n$ is $n$. This induces filtrations on $A \otimes C$ and $B \otimes C$, in which
\[	F_{\le n}(A \otimes C) = \sum_{q \le n} A \otimes C_q,	\]
and likewise for $B \otimes C$. Consider the map
\[	
\nu = (1 \otimes \pi)\Delta\phi: A \otimes C \to B \otimes C.
\]
Using the comultiplicativity of $s$, we see that 
\[	\nu(1 \otimes \beta_n) = \sum_{i=0}^n s(\beta_i) \otimes \beta_{n-i}.	\]
Furthermore, since $\nu$ is a left $A$-module map, it preserves the filtration. Thus, there is an induced map $\ol{\nu}$ on associated graded objects. However, as $C$ is the direct sum of the $C_q$, the associated graded objects are simply $A \otimes C$ and $B \otimes C$. Once again, one computes that
\[	
\ol{\nu}(1 \otimes \beta_n) = 1 \otimes \beta_n.
\]
As $\ol{\nu}$ is a left $A$-module map, we can identify it with
\[	
i \otimes 1:A \otimes C \to B \otimes C.
\]
Since $C$ is flat over $\Z_p$, this map is injective. Thus $\nu$ is injective, so $\phi$ is injective, as desired. 

For surjectivity, we use a version of \cite[Proposition 1.6]{MilnorMoore}. Filter $A$ as follows: the elements of filtration $\ge s$ are the polynomials in $f, \theta(f), \theta_2(f), \dotsc$ all of whose terms are of degree $\ge s$. Giving $B$ the analogous filtration, the map $i:A \to B$ is a filtered $A$-module map, and the counit $\epsilon:A \to \Z_p$ kills the ideal of positively filtered elements. The $A$-module structure on $C$ factors through $\epsilon$, and we give $C$ the trivial filtration $C = C_{\ge 0} = C_{\ge 1} = \dotsb$. Then $\pi:B \to C$ is also filtered.

\vspace{1em}
\noindent \textbf{Claim 1.} Let $M$ be a nonnegatively filtered $A$-module. Then $M = 0$ iff $\Z_p \otimes_A M = 0$.

\vspace{1em}
Indeed, if $M$ is nonzero, then it has a nonzero element $x$ of lowest possible filtration, say $s$. But the kernel of $M \to \Z_p \otimes_A M$ is precisely $A_{> 0}\cdot M$, so if $\Z_p \otimes_A M = 0$, then $x$ is an $A$-multiple of an element of lower filtration. 

\vspace{1em}
\noindent \textbf{Claim 2.} Let $g:M_1 \to M_2$ be a filtered $A$-module map, where $M_1$ and $M_2$ are nonnegatively filtered. Then $g$ is surjective iff 
\[	\Z_p \otimes_A g: \Z_p \otimes_A M_1 \to \Z_p \otimes_A M_2	\]
is surjective.

\vspace{1em}
The direction $(\Rightarrow)$ holds because the tensor product is right exact. For the direction $(\Leftarrow)$, let $N = \mathrm{coker}(g)$. The $A$-module $N$ receives a filtration in an evident way. Again using right exactness of the tensor product, we have that 
\[	\Z_p \otimes_A N = \mathrm{coker}(\Z_p \otimes_A g).	\]
If $\Z_p \otimes_A g$ is surjective, then $\Z_p \otimes_A N = 0$, so $N = 0$ by Claim 1. (Since completion is neither left nor right exact in general, we need to work with the uncompleted tensor product here.)

Finally, $\phi: A \otimes C \to B$ is a filtered $A$-module map whose source and target are nonnegatively filtered. We have
\[	\Z_p \otimes_A \phi = \mathrm{id}:C \to C = \Z_p \otimes_A B.	\]
By Claim 2, $\phi$ is surjective.
\end{proof}\begin{proof}[Proof of \Cref{thm: KO homology of Tzeta}]
We have already constructed an isomorphism of $KO_0KO$-comodules
\[	\TT(f) \otimes KO_0KO \to KO_0T_\zeta.	\]
To extend this to a map
\[	\TT(f) \otimes KO_*KO = KO_* \otimes \TT(f) \otimes KO_0KO \to KO_*T_\zeta,	\]
one has identify the image of $KO_*$ in $KO_*T_\zeta$, which will consist of elements which are invariant under the Adams operations. First, suppose that $p > 2$. Since $g \in \Z_p^\times$ maps to a topological generator of $\Z_p^\times/\mu$, we have $g^{p-1} \in 1 + p\Z_p$. Write $g^{p-1} = 1+h$ where $h \in p\Z_p$. Then the series
\[	g^{-b(p-1)} := (1+h)^{-b} = \sum_{n \ge 0} \binom{-b}{n}h^n	\]
converges in $\TT(b)$. This uses the classical fact that
\[	\lim_{n\to\infty} n - v_p(n!) = \infty.	\]
In $KO_*T_\zeta = KO_* \otimes \TT(b)$,
\[	\psi^g(g^{-b(p-1)}v_1) = (1+h)^{-(b+1)}\cdot g^{p-1}v_1 = g^{-b(p-1)}v_1.	\]
Thus, writing $\widetilde{v_1} = g^{-b(p-1)}v_1 \in KO_{2(p-1)}T_\zeta$, we see that multiplication by $\widetilde{v_1}^k$ induces an isomorphism of $KO_0KO$-comodules 
\[	KO_0T_\zeta \stackrel{\sim}{\to} KO_{2(p-1)k}T_\zeta.	\]
As $KO_0T_\zeta$ is an extended comodule, the same follows for $KO_*T_\zeta$, and we obtain
\[	\pi_*T_\zeta = (KO_*T_\zeta)^{\Z_p^\times/\mu} = \Z_p[\widetilde{v_1}^{\pm 1}] \otimes \TT(f).	\]
The isomorphism with $KO_* \otimes \TT(f)$ is given by mapping $\widetilde{v_1}$ to $v_1$.

Now suppose that $p = 2$, in which case $KO_*$ is generated by $\eta \in KO_1$, $v = 2u^2 \in KO_4$, and $w = u^4 \in KO_8$, where $u \in K_2$ is the Bott element. We have that $g^2 = 1 + h$ where $h \in 4\Z_2$. Again, this means that the series $g^{-2b} = (1+h)^{-b}$ converges, and we can define $\widetilde{v} = g^{-2b}v$, $\widetilde{w} = g^{-4b}w$. By the same arguments, $KO_{4*}T_\zeta$ is an etended comodule. To deal with the rest, we note that
\[	KO_{8k+1}T_\zeta \cong KO_{8k+2}T_\zeta \cong KO_0T_\zeta \otimes_{\Z_2} \F_2	\]
as $KO_0KO$-comodules. Tensoring the exact sequence
\[	0 \to \pi_0T_\zeta \to KO_0T_\zeta \stackrel{\psi^g-1}{\to} KO_0T_\zeta \to 0	\]
with $\F_2$ and noting that $KO_0T_\zeta$ is flat over $\Z_2$, we obtain the desired result.
\end{proof}

\begin{rmk}\label{rmk: Hopkins mistake}
	As we mentioned earlier, Hopkins' argument from \cite{K1localrings} has errors. In particular, he rightly claims that the map 
	\[
	\begin{tikzcd}
		\TT(f)\otimes \Cts(\Z_p, \Z_p)\arrow[r,"i\otimes s"] &\TT(b)\otimes \TT(b)\arrow[r,"\mathrm{mult}"] & \TT(b).
	\end{tikzcd}
	\]
	However, he argues this by asserting that the inverse to this map is given by 
	\[
	\begin{tikzcd}
		\TT(b)\arrow[r,"\Delta"] & \TT(b)\otimes \TT(b)\arrow[r,"(1-s\circ \pi)\otimes \pi"] & \TT(f)\otimes \Cts(\Z_p, \Z_p).
	\end{tikzcd}
	\]
	But this map simply cannot be the inverse, indeed it is not even injective. To see this, let $\beta_n$ denote the $n$th binomial coefficient function. The section $s$ is a map of coalgebras and the diagonal on the $\beta_n$ satisfy the Cartan formula. Thus 
	\[
	\Delta(s\beta_n) = \sum_{i+j=n}s(\beta_i)\otimes s(\beta_j).
	\]
	Thus, under the above map, one computes that 
	\[
	s(\beta_n)\mapsto \sum_{i+j=n}\pi(s(\beta_i))\otimes (1-s\pi)(s\beta_j).
	\]
	Since $s$ is a section, $\pi s = 1$. Note that 
	\[
	(1-s\pi)(s(\beta_j)) = s(\beta_j) - s\pi s(\beta_j) = s(\beta_j)-s(\beta_j)=0.
	\]
	Note that this includes the case when $j=0$, in which case $\beta_j = \beta_0 = 1$. Thus the above map has a nontrivial kernel, and so is not injective. 
\end{rmk}

\section{Co-operations for \texorpdfstring{$T_\zeta$}{Tzeta}}\label{sec: coops for cone on zeta}

We saw in \Cref{prop: homology of cone on zeta} that $KO_*T_\zeta \cong KO_* \otimes \TT(b)$. As $\TT(b)$ is a completion of a polynomial ring, $KO_*T_\zeta$ is pro-free over $KO_*$. Moreover, we have an equivalence of $KO$-modules in $\Sp$, 
\[
KO\wedge T_\zeta \wedge T_\zeta\simeq (KO\wedge T_\zeta)\wedge_{KO} (KO\wedge T_\zeta).
\] 
So it follows from \Cref{prop: pro-free tensor} that,
\begin{equation}\label{eq: KO-homology Tzeta coops}
	KO_*(T_\zeta \sma T_\zeta) \cong KO_*T_\zeta \otimes_{KO_*} KO_*T_\zeta \cong KO_* \otimes \TT(b, b').
\end{equation}
Recall that the $KO_*KO$-comodule structure is given by an action of the group $\Z_p^\times/\mu$. In this case, the action comes from the diagonal action on the two tensor factors, so that
\begin{equation*}
	\psi^g(b) = b+1, \quad \psi^g(b') = b'+1.
\end{equation*}

As we saw in the previous section, the computation of $\pi_*T_\zeta$ followed from knowing that $KO_*T_\zeta$ was an extended comodule. The same strategy allows us to compute the co-operations algebra $\pi_*(T_\zeta \sma T_\zeta)$, after the following lemma.

\begin{lem}\label{lem: tensor product extended}
A tensor product of extended $KO_*KO$-comodules is extended. Moreover, we have
\begin{gather*}
	\Hom_{\Comod_{KO_*KO}^\wedge}(KO_*, (KO_*KO \otimes_{KO_*} M_*) \otimes_{KO_*} (KO_*KO \otimes_{KO_*} N_*)) \\
	\cong KO_*KO \otimes_{KO_*} M_* \otimes_{KO_*} N_* 
	\cong \Cts(\Z_p^\times/\mu, \Z_p) \otimes_{\Z_p} M_* \otimes_{KO_*} N_*.
\end{gather*}
\end{lem}

\begin{proof}
One immediately reduces to the case $M_* = N_* = KO_*$. Abbreviating the group $\Z_p^\times/\mu$ as $G$, we have
\[	KO_*KO \otimes_{KO_*} KO_*KO \cong \Cts(G, KO_*) \otimes_{KO_*} \Cts(G, KO_*) \cong \Cts(G \times G, KO_*),	\]
By the results of \Cref{subsec: Hopf algebroids K KO}, this is an extended $KO_*KO$-comodule iff it is an induced $G$-module, where $G$ acts by precomposition with the diagonal action on $G \times G$.

Define
\begin{align*}
	&m: G \times G \to G, \quad (g_1, g_2) \mapsto g_1g_2^{-1}, \\
	&p: G \times G \to G, \quad (g_1, g_2) \mapsto g_2.
\end{align*}
Consider the map
\[	
\varphi:= m \times p: G \times G \to G \times G; (g_1,g_2)\mapsto (g_1g_2^{-1}, g_2).
\]
We regard the target as a $G$-set via
\[
h\cdot (g_1,g_2) = (g_1,h g_2)
\]
and the source as a $G$-set via the diagonal action. Then the map is a continuous map of $G$-sets. Note also that $\varphi$ is a bijection, with inverse
\[
\varphi^{-1}(g_1,g_2) = (g_1g_2,g_2).
\] 
and hence $\varphi$ is an isomorphism of $G$-sets. Thus the induced map
\[	
\varphi^*=m^* \otimes p^*: \Cts(G, KO_*) \otimes \Cts(G, KO_*) \to \Cts(G \times G, KO_*)
\]
is an isomorphism of $KO_*$-modules. As $\varphi$ is an equivariant map, it is also a map of $G$-modules, where the first copy of $\Cts(G, KO_*)$ is given the trivial $G$-action and the second is given the usual action. This implies that $\Cts(G \times G, KO_*)$, and more precisely, of the form
\[	\Cts(G \times G, KO_*) \cong KO_*KO \otimes_{KO_*} \Cts(G, KO_*),	\]
where the second tensor factor has trivial coaction and is mapped into $\Cts(G \times G, KO_*)$ via $m^*$. By \Cref{prop: relative ext facts}, the primitives of $\Cts(G \times G, KO_*)$ is just this second tensor factor.
\end{proof}

\begin{rmk}
	The above proof relies, in an essential way, on the fact that the group $G = \Z_p^\times/\mu$ is an abelian group. 
\end{rmk}
 
In the following, we will frequently use $x$ and $\overline{x}$ to denote the image of $x$ along respectively the left and right units of a Hopf algebroid.

\begin{thm} 
There is an isomorphism of $\theta$-algebras
\[	\pi_*(T_\zeta \sma T_\zeta) \cong KO_* \otimes \TT(f,\overline{f},\ell)/(\psi^p(\ell) - \ell - f + \overline{f}).	\]
\end{thm}

\begin{proof}
As $KO_*T_\zeta$ is $KO_*$-pro-free, we have
\[	KO_*(T_\zeta \sma T_\zeta) \cong KO_*(T_\zeta) \otimes KO_*(T_\zeta)	\]
as $KO_*KO$-comodules. We saw in the proof of \Cref{thm: KO homology of Tzeta} that $KO_*T_\zeta$ is an extended comodule. The lemma then implies that $KO_*(T_\zeta \sma T_\zeta)$ is extended, and that there is an additive isomorphism
\begin{equation*}
\begin{split}
	\pi_*(T_\zeta \sma T_\zeta) = \Hom_{\Comod_{KO_*}^\wedge}(KO_*, KO_*(T_\zeta \sma T_\zeta)^{\Z_p^\times/\mu}) \\
	 \cong \pi_*T_\zeta \otimes_{KO_*} \pi_*T_\zeta \otimes \Cts(\Z_p^\times/\mu,\Z_p) \\
	\cong KO_* \otimes \TT(f,\overline{f}) \otimes \Cts(\Z_p, \Z_p).
\end{split}
\end{equation*}
Here $f$ and $\overline{f}$ come from the left and right copies of $\pi_0T_\zeta$ respectively.

Note that, as the isomorphism $KO_*T_\zeta \cong \TT(f) \otimes KO_*KO$ of \Cref{thm: KO homology of Tzeta} is an isomorphism of comodules but not of comodule algebras; the above isomorphism is only additive. We can nevertheless identify the multiplicative structure on $\pi_*(T_\zeta \sma T_\zeta)$ by locating the primitive elements identified above inside the ring
\[	KO_*(T_\zeta \sma T_\zeta) = KO_* \otimes \TT(b, \overline{b}).	\]
In fact, the theta-algebra $\TT(f, \overline{f})$ is just that generated by $f = \psi^p(b) - b$ and $\ol{f} = \psi^p(\ol{b}) - b$ inside $KO_*(T_\zeta \sma T_\zeta)$. Likewise, there is a primitive copy of $KO_*$ inside $KO_*(T_\zeta \sma T_\zeta)$, namely that generated by the left unit on $\widetilde{v_1}$ (or by the left unit on $\eta$, $\widetilde{v}$, and $\widetilde{w}$ if $p=2$).

We still have to identify the $\Cts(\Z_p,\Z_p)$ factor. The proof of \Cref{lem: tensor product extended} tells us that, under the isomorphism
\[	KO_0(T_\zeta \sma T_\zeta) \cong \TT(f,\overline{f}) \otimes \Cts(\Z_p \times \Z_p, \Z_p),	\]
this factor is precisely
\begin{equation}\label{antidiagonal}
	\{1 \otimes \phi \in \TT(f,\overline{f}) \otimes \Cts(\Z_p \times \Z_p, \Z_p):f(x,y) = f(0,y-x)\}.
\end{equation}
The submodule $1 \otimes \Cts(\Z_p \times \Z_p, \Z_p)$ is the image of
\[	s \otimes \overline{s}:\Cts(\Z_p, \Z_p) \otimes \Cts(\Z_p, \Z_p) \to KO_0(T_\zeta \sma T_\zeta),	\]
where $s$ is as defined in \Cref{s}. That is,
\[	(s \otimes \overline{s})(\beta_n \otimes \beta_m) = \lambda^n(b)\lambda^m(\overline{b}).	\]
By Mahler's theorem, the submodule of $f \in \Cts(\Z_p \times \Z_p, \Z_p)$ satisfying the condition of \eqref{antidiagonal} is spanned by 
\[	(x,y) \mapsto \binom{y-x}{n}.	\]
Thus, the invariant $\Cts(\Z_p, \Z_p)$ factor in $KO_0(T_\zeta \sma T_\zeta)$ is spanned by
\[	\lambda^n(b - \overline{b}).	\]

In particular, the sub-$\lambda$-algebra of $KO_0(T_\zeta \sma T_\zeta)$ generated by $b-\overline{b}$ contains this $\Cts(\Z_p,\Z_p)$. But this is the same as the sub-$\theta$-algebra generated by $b-\overline{b}$. Let
\[	\ell = b - \overline{b}.	\]
The formula $\psi^p(b) - b = f$, and the analogous one for $\overline{f}$, show that
\[	f - \overline{f} = \psi^p(\ell) - \ell.	\]
Thus, there is an epimorphism
\begin{equation}\label{Tzeta cooperations epi}
	\TT(f,\overline{f},\ell)/(\psi^p(\ell) - \ell - f + \overline{f}) \surj \pi_0(T_\zeta \sma T_\zeta).
\end{equation}

To see that this is an isomorphism, note that \Cref{prop: homology of cone on zeta} implies that
\[	\pi_0(T_\zeta \sma T_\zeta) \cong \pi_0T_\zeta \otimes \TT(x)	\]
as $\pi_0(T_\zeta)$-modules. That is, it is a free $\theta$-algebra on two generators. But the left-hand side of \Cref{Tzeta cooperations epi} is free on the generators $f$ and $\ell$, and any nontrivial quotient of it would not be free on two generators. Thus, we have
\[	\pi_0(T_\zeta \sma T_\zeta) = \TT(f,\overline{f},\ell)/(\psi^p(\ell) - \ell - f + \overline{f}).	\]
This concludes the proof.

%
\end{proof}

\section{\texorpdfstring{$K(1)$}{K(1)}-local \texorpdfstring{$\tmf$}{tmf}}\label{sec: tmf}

We continue to work $K(1)$-locally, and fix $p = 2$ or 3, so that $j = 0$ is the unique supersingular $j$-invariant. It is simple to extend this story to larger primes with a single supersingular $j$-invariant; slightly more complicated to extend it to other primes; but in neither case is it quite as interesting. As in section 4, the statements in this section are due to \cite{K1localrings}.

\begin{prop}
For any $x \in KO_0\tmf$ such that $\psi^g(x) = x+1$, there is a unique homotopy class of $E_\infty$ maps $T_\zeta \to \tmf$ sending $b \in KO_0T_\zeta$ to $x$.
\end{prop}

\begin{proof}
Clearly, any map $T_\zeta \to \tmf$ acts this way on $KO$-homology. Conversely, since $\pi_{-1}\tmf = 0$, the set of homotopy classes of $E_\infty$ maps $T_\zeta \to \tmf$ is parametrized by 
\[	\pi_0\tmf = \mathrm{Maps}_{\theta\mathsf{-Alg}}(\TT(f),\pi_0\tmf).	\]
Since $KO_0T_\zeta$ is the induced $KO_0KO$-comodule on $\pi_0T_\zeta$, any such $\theta$-algebra map extends uniquely to a $\psi$-$\theta$-algebra map
\[	KO_0T_\zeta \to KO_0\tmf	\]
and thus to
\[	KO_*T_\zeta \to KO_*\tmf.	\]
\end{proof}

In particular, we can pick
\[	g=3\text{ and }x = -\frac{\log{c_4/w}}{\log 3^4}\text{ at }p=2, 	\]
\[	g=2\text{ and }x = -\frac{\log{c_6/v_1^3}}{\log 2^6}\text{ at }p=3.	\]

\begin{prop}[{\cite[7.1]{K1localrings}}]\label{f congruent to j inverse}
Let $b$ be as above and let $f = \psi^p(b) - b$. Then $f \equiv j^{-1}$ mod $p$, and as an element of $\Z_p[j^{-1}]$, $f$ has constant term zero. Thus, the map $\Z_p[f] \to \Z_p[j^{-1}]$ is an isomorphism.
\end{prop}

\begin{proof}
This is a calculation using $q$-expansions. See \cite[7.1]{K1localrings}.
\end{proof}

It follows that the map $q:T_\zeta \to \tmf$ induces a surjective map on $\pi_0$,
\[	q:\TT(f) \surj \Z_p[j^{-1}].	\]
Thus, $\theta(f)$ maps to some completed polynomial in $j^{-1}$. Since $f \equiv j^{-1}$ mod $p$, this can also be written as a completed polynomial in $f$, say $h(f)$. It follows that the kernel of $q$ is the $\theta$-ideal generated by $\theta(f) - h(f)$.

\begin{lem}
The map of $\theta$-algebras $F:\TT(x) \to \TT(b)$ sending $x$ to $\theta(f) - h(f)$ makes $\TT(b)$ into a pro-free $\TT(x)$-module.
\end{lem}

\begin{proof}
This is similar to \Cref{i injective flat}. Again, let us write 
\[	x_i = \theta_i(x), \quad b_i = \theta_i(b), \quad i \ge 0.	\]
(See \Cref{thm: free theta-algebra} for $\theta_i$.) We will prove by induction that
\begin{equation}\label{x congruence}
	F(x_i) = b_{i+1}^p - b_{i+1} \mod{(p,b_0, \dotsc, b_i)}.
\end{equation}
When $i = 0$,
\begin{align*}
	F(x) &= \theta(f) - h(f) \\
	&= \theta(\psi^p(b) - b) - h(\psi^p(b) - b) \\
	&= \frac{1}{p}(\psi^{p^2}(b) - \psi^p(b) - (\psi^p(b) - b)^p) - h(\psi^p(b) - b) \\
	&= \frac{1}{p}(b_0^{p^2} - b_0^p + p(b_1^p - b_1) + p^2b_2 - (b_0^p - b_0 + pb_1)^p) - h(b_0^p - b_0 + pb_1) \\
	&\equiv b_1^p - b_1 + pb_2 - p^{p-1}b_1^p - h(pb_1) \pmod{b_0} \\
	&\equiv b_1^p - b_1 \pmod{(p, b_0)}.
\end{align*}
Suppose that we have proved \eqref{x congruence} for $i = 0, \dotsc, n-1$. Then for these values of $i$,
\[	F(x_{i}) \equiv 0 \mod{(p, b_0, \dotsc, b_{i+1})},	\]
and so
\begin{equation}\label{x congruence 2}
	p^iF(x_{n-i})^{p^{n-i}} \equiv 0 \mod{(p^{n+1}, b_0, \dotsc, b_{i+1})}.
\end{equation}
We also have
\begin{align*}
	\theta(\psi^{p^{n+1}}(b) - \psi^{p^n}(b)) &= \frac{1}{p}(\psi^{p^{n+2}}(b) - \psi^{p^{n+1}}(b) - (\psi^{p^{n+1}}(b) - \psi^{p^n}(b))^p) \\
	&\equiv \frac{1}{p}(p^{n+1}(b_{n+1}^p - b_{n+1}) + p^{n+2}b_{n+2} - (p^{n+1}b_{n+1})^p) \mod{(b_0, \dotsc, b_n)} \\
	&\equiv p^n(b_{n+1}^p - b_{n+1}) \mod{(p^{n+1}, b_0, \dotsc, b_n)}.
\end{align*}
Finally,
\[	h(\psi^{p^{n+1}}(b) - \psi^{p^n}(b)) \equiv 0 \mod{(p^{n+1}, b_0, \dotsc, b_n)}	\]
because $h$ is a completed polynomial over $\Z_p$. Putting this all together,
\begin{align*}
	\psi^{p^n}(F(x)) &= \psi^{p^n}(\theta(\psi^p(b) - b) - h(\psi^p(b) - b)) \\
	&= \theta(\psi^{p^{n+1}}(b) - \psi^{p^n}(b)) - h(\psi^{p^{n+1}}(b) - \psi^{p^n}(b)) \\
	&\equiv p^n(b_{n+1}^p - b_{n+1}) \mod{(p^{n+1}, b_0, \dotsc, b_n)}.
\end{align*}
The left-hand side is congruent to $p^nF(x_n)$ modulo this ideal by \eqref{x congruence 2}, which proves \eqref{x congruence}.

It follows that the map
\[	\F_p[b_0,x_0,x_1,\dotsc] \to \F_p[b_0,b_1,b_2, \dotsc]	\]
makes the target into a free module over the source, by the same argument as in \Cref{i injective flat}. But $\F_p[b_0, x_0, x_1, \dotsc]$ is clearly free over $\F_p[x_0, x_1, \dotsc]$. By \Cref{lem: pro-free mod p}, $\TT(b)$ is pro-free over $\TT(x)$. This finishes the proof of the lemma.
\end{proof}

\begin{thm}[{\cite[7.2]{K1localrings}}]\label{tmf presentation}
There is a homotopy pushout square of $K(1)$-local $E_\infty$ rings,
\[	\xymatrix{ \PP(S^0) \ar[d]_{\theta(f) - h(f)} \ar[r]^0 & S^0 \ar[d] \\ T_\zeta \ar[r]_q & \pushoutcorner \tmf. }	\]
\end{thm}

\begin{proof}
Let $Y$ be the homotopy pushout of the above square, so 
\[	Y \simeq T_\zeta \wedge_{\PP(S^0)} S^0.	\]
Since $\theta(f) = h(f)$ in $\pi_0\tmf$, there is a map $Y \to \tmf$, which we will show is an isomorphism on homotopy groups.

We note that $KO_*\PP(S^0) \to KO_*T_\zeta$ is precisely the map of the previous lemma, tensored by $KO_*$. By \Cref{lem: pro-free base change}, $KO_*T_\zeta$ is pro-free over $KO_*\PP(S^0)$. Then by \Cref{prop: pro-free tensor} and the previous lemma, we have the K\"unneth formula,
\[	KO_*Y = KO_*T_\zeta \otimes_{KO_*\PP(S^0)} KO_* \cong KO_*T_\zeta \otimes_{KO_0\PP(S^0)} \Z_p.	\]
By \Cref{T(b) induced} and the proof of \Cref{thm: KO homology of Tzeta}, we have an isomorphism
\[	KO_*T_\zeta \cong \pi_*T_\zeta \otimes KO_0KO	\]
as $\pi_*T_\zeta$-modules and $KO_0KO$-comodules. Since $KO_0\PP(S^0) \to KO_0T_\zeta$ factors through $\pi_0T_\zeta$, we likewise have
\[	KO_*Y = KO_*T_\zeta \otimes_{KO_0\PP(S^0)} \Z_p \cong (\pi_*T_\zeta \otimes_{KO_0\PP(S^0)} \Z_p) \otimes KO_0KO	\]
as $KO_0KO$-comodules. That is, $KO_*Y$ is an induced comodule, and
\[	\pi_*Y = \pi_*T_\zeta \otimes_{KO_0\PP(S^0)} \Z_p = KO_* \otimes \TT(f)/(\theta(f) - h(f)) = \Z_p[f] = \pi_*\tmf.	\]
(Here the quotient is by the $\theta$-ideal generated by $\theta(f) - h(f)$.)
\end{proof}

\begin{cor}
There is an $E_\infty$ map $r:\tmf \to KO$.
\end{cor}

\begin{proof}
One has an $E_\infty$ map $T_\zeta \to KO$, which by arguments similar to the ones above fits into a pushout square of $E_\infty$ rings
\[	\xymatrix{ \PP(S^0) \ar[d]_f \ar[r] & S^0 \ar[d] \\ T_\zeta \ar[r] & KO. }	\]
The left-hand vertical map sends the $\theta$-algebra generator $x$ of $KO_0\PP(S^0)$ to $f = \psi^p(b) - b \in KO_0T_\zeta$. There is an $E_\infty$ factorization
\[	\xymatrix@1{ \PP(S^0) \ar[rr]^{\theta(x) - h(x)} \ar@/_1em/[rrr]_{\theta(f) - h(f)} & & \PP(S^0) \ar[r]^f & T_\zeta. }	\]
This induces a map from the $E_\infty$ cofiber of the composite, namely $\tmf$, to the $E_\infty$ cofiber of the right-hand map, namely $KO$.
\end{proof}

On coefficients, the map $r$ is just
\[	KO_*[j^{-1}] \to KO_*: \quad j^{-1} \mapsto 0.	\]
Despite the obvious splitting of $r$ at the level of coefficients, it is not clear whether or not there exists an $E_\infty$ map from $KO$ to $\tmf$.

\section{Co-operations for \texorpdfstring{$K(1)$}{K(1)}-local \texorpdfstring{$\tmf$}{tmf}}\label{sec: coops for tmf}

The preceding \Cref{tmf presentation} gave a presentation of $K(1)$-local $\tmf$ in terms of finitely many $E_\infty$ cells. We can now use this presentation to describe the $K(1)$-localization of $\tmf \sma \tmf$.

\begin{thm}
The homotopy groups of $\tmf \sma \tmf$ are given by
\begin{align*}
	\pi_*(\tmf \sma \tmf) &= KO_* \otimes  \Z_p[f, \ol{f}] \otimes \TT(\ell)/(\psi^p(\ell) - \ell - f + \ol{f}).
\end{align*}
\end{thm}

\begin{proof}
Write $F:\PP(S^0) \to T_\zeta$ for the map sending the generator $x \in KO_0\PP(S^0) = \TT(x)$ to $\theta(f) - h(f)$. We saw in the previous section that $F$ induces a pro-free map on $KO$-homology, and that
\[	
\tmf = S^0 \sma_{\PP(S^0)}^F T_\zeta.
\]
Therefore,
\[	
\tmf \sma \tmf = (S^0 \sma_{\PP(S^0)}^F T_\zeta) \sma (T_\zeta\,\, \leftidx{^F}{\!\!\sma}{_{\PP(S^0)}} S^0)\simeq (S^0\wedge S^0)\wedge_{\PP(S^0)\wedge \PP(S^0)}^{F\wedge F}(T_\zeta\wedge T_\zeta).
\]
Since $F:KO_*\PP(S^0) \to KO_*T_\zeta$ is flat, this has $KO$-homology
\begin{align*}
	KO_*(\tmf \sma \tmf) &= (KO_*(T_\zeta) \otimes_{KO_*} KO_*(T_\zeta))\otimes_{KO_*(\PP S^0\wedge \PP S^0)}^{F\otimes F} KO_* \\
	&= KO_* \otimes \TT(b,\ol{b}) \otimes_{F\otimes F, \TT(x,\overline{x})} \Z_p.
\end{align*}
But $KO_* \otimes \TT(b,\ol{b})$ is an extended $KO_*KO$-comodule, and $(F \otimes F)$ factors through its fixed points, which are
\[	\pi_*(T_\zeta \sma T_\zeta) = \TT(f,\ol{f},\ell)/(\psi^p(\ell) - \ell - f + \ol{f}).	\]
By the arguments of \Cref{tmf presentation}, $KO_*(\tmf \sma \tmf)$ is also extended, and
\begin{align*}
	\pi_*(\tmf \sma \tmf) &= KO_* \otimes \TT(f,\ol{f},\ell)/(\theta(f) - h(f), \theta(\ol{f}) - h(\ol{f}), \psi^p(\ell) - \ell - f + \ol{f}) \\
	&\cong KO_*[f,\ol{f}] \otimes \TT(\ell)/(\psi^p(\ell) - \ell - f + \ol{f}).
\end{align*}
\end{proof}

\begin{rmk}\label{rmk: generator lambda}
For a more modular presentation of this ring, recall from \Cref{f congruent to j inverse} that $f = \alpha(j^{-1})$ for some invertible power series $\alpha \in \Z_p[j^{-1}]$. Thus, $KO_*[f,\ol{f}] = KO_*[j^{-1},\ol{j^{-1}}]$. Letting
\[	\lambda = \alpha(\ell),	\]
we can equivalently write
\[	\tmf_*\tmf = KO_* \otimes \Z_p[j^{-1}, \ol{j^{-1}}] \otimes \TT(\lambda)/(\psi^p(\lambda) - \lambda - j^{-1} + \ol{j^{-1}}).	\]
\end{rmk}

We now consider the Hopf algebroid for $K(1)$-local $\tmf$.

To obtain $\tmf_*\tmf$ from $T_{\zeta,*}T_\zeta$, we take the $\theta$-algebra quotient induced by the relation $\theta(f) = h(f)$, and the same relation for $\ol{f}$. We obtain
\begin{equation}\label{eq: tmf coops}
	\tmf_*\tmf = \tmf_* \otimes \TT(\ell)/(\theta(f + \ell - \psi^p(\ell)) - h(f + \ell - \psi^p(\ell))),
\end{equation}
where again the quotient is by a $\theta$-ideal.

This formula should be compared to the analogous one for $K(1)$-local $KO$-cooperations: as a $\theta$-algebra, we have
\[	KO_*KO \cong KO_* \otimes \Cts(\Z_p,\Z_p) \cong KO_* \otimes \TT(b)/(\psi^p(b) - b),	\]
where the last isomorphism follows from \Cref{j iso}. That is, $KO_*KO$ is generated as a $\theta$-algebra over $KO_*$ by a single generator $b$, with an algebraic relation between $b$ and $p\theta(b)$. Likewise, $\tmf_*\tmf$ is generated over $\tmf_*$ by a single generator $\ell$, with an algebraic relation over the coefficient ring $\Z_p[f]$ that relates $\ell$, $\theta(\ell)$, and $p\theta_2(\ell)$. One can think of this as a second-order version of the $\theta$-algebraic structure underlying $KO_*KO$.

Now, the coalgebra presentation $KO_*KO = KO_* \otimes \Cts(\Z_p,\Z_p)$ is in fact more useful than the algebra presentation $KO_*KO = KO_* \otimes \TT(b)/(\psi^p(b) - b)$. The former allows the simple computation of the $KO$-based Adams spectral sequence for arbitrary $X$: its $E_2$ page is just the group cohomology
\[	E_2 = H^*_{cts}(\Z_p^\times/\mu, KO_*X),	\]
and as this is concentrated on two lines, the spectral sequence collapses at $E_2$ and always converges. As it turns out, very similar statements are true for $\tmf$.

\begin{prop}
The left unit $\tmf_* \to \tmf_*\tmf$ is pro-free.
\end{prop}
 
\begin{proof}
While one can prove pro-freeness algebraically by applying \Cref{lem: pro-free base change} and \Cref{lem: pro-free mod p} to the formula \eqref{eq: tmf coops}, it is easier to use Laures's \cite[Corollary 3]{Laures2004}, which gives an additive equivalence of homology theories
\[	\tmf_*X \cong KO_*X[j^{-1}],	\]
and correspondingly an additive equivalence of $K(1)$-local spectra
\[	\tmf \simeq KO[j^{-1}] = \bigvee_{n=1}^\infty KO.	\]
Thus, to show that $\tmf_*\tmf$ is pro-free over $\tmf_*$, it suffices to show that $KO_*\tmf$ is pro-free over $KO_*$. From \Cref{lem: pro-free wedge of spheres}, one observes that the property of $KO_*X$ being pro-free over $KO_*$ is closed under coproducts. As $KO_*KO$ is pro-free over $KO_*$, and $\tmf$ is a coproduct of copies of $KO$, $KO_*\tmf$ is also pro-free.
\end{proof}
 
\begin{cor}
There is an $L$-complete Hopf algebroid $(\tmf_*, \tmf_*\tmf)$. For any $K(1)$-local spectrum $X$, the $K(1)$-local Adams spectral sequence based on $\tmf$ is conditionally convergent and takes the form
 \[	E_2 = \Ext_{\tmf_*\tmf}(\tmf_*, \tmf_*X) \Rightarrow \pi_*X.	\]
\end{cor}

\begin{proof}
Since $\tmf_* \to \tmf_*\tmf$ is pro-free, one has an $L$-complete Hopf algebroid by \Cref{Lcomplete Hopf algebroid}. Then by \Cref{prop: K1local ASS}, the $E_2$ page of the Adams spectral sequence has the form described.

To establish convergence, one needs to show that $X$ is $K(1)$-local $\tmf$-nilpotent. Recall from \cite[Appendix 1]{DH} and \cite{Bousfield79} that this is the largest class of $K(1)$-local spectra containing $\tmf$ and closed under retracts, cofibers, and $K(1)$-local smash products with arbitrary spectra. Now, multiplication by $j^{-1}$ gives a cofiber sequence
\[	\tmf \stackrel{j^{-1}}{\to} \tmf \to KO,	\]
so that $KO$ is $K(1)$-local $\tmf$-nilpotent. The cofiber sequence
\[	S \to KO \to KO	\]
then shows that the sphere is $K(1)$-local $\tmf$-nilpotent. This clearly implies that an arbitrary spectrum is $K(1)$-local $\tmf$-nilpotent.
\end{proof}

We can now prove \Cref{thm: ASS}.

\begin{thm}
There is a natural isomorphism
\[	\Ext_{\tmf_*\tmf}(\tmf_*, \tmf_*) \cong \Ext_{KO_*KO}(KO_*, KO_*).	\]
\end{thm}

\begin{proof}
The ring map $\tmf \to KO$ induces a map of Hopf algebroids,
\[	(\tmf_*, \tmf_*\tmf) \to (KO_*, KO_*KO).	\]
The map $\tmf_* \to KO_*$ sends $j^{-1}$ to zero, and thus sends $f = j^{-1} + O(pj^{-1}, j^{-2})$ to zero as well. We have
\begin{multline*}
	KO_* \otimes_{\tmf_*} \tmf_*\tmf \otimes_{\tmf_*} KO_* \\
	= KO_* \otimes_{\tmf_*} (KO_* \otimes \Z_p[f, \ol{f}] \otimes \TT(\ell)/(f - \ol{f} - \psi^p(\ell) + \ell)) \otimes_{\tmf_*} KO_*  \\
	\cong KO_* \otimes \TT(\ell)/(\psi^p(\ell) - \ell).
\end{multline*}
We need to identify the image of $\ell$ in $KO_*KO$. Consider the commuting square
\[	\xymatrix{ \tmf_*\tmf \ar[r] \ar[d] & KO_*(\tmf \sma \tmf) \ar[d] \\ KO_*KO \ar[r] & KO_*(KO \sma KO). }	\]
The horizontal maps are both inclusions of $KO_*KO$-primitives, and, in particular, injective. Going from $\tmf_*\tmf$ to $KO_*(KO \sma KO)$ around the top right corner sends $\ell$ to $b - \ol{b}$, where we recall that
\[	b \in KO_0KO \cong \Cts(\Z_p, \Z_p)	\]
is the identity map on $\Z_p$. Using the Hopf algebroid formulas found in \Cref{thm: K Hopf algebroid}, together with the group isomorphism $\Z_p^\times/\mu \cong \Z_p$, we have
\[	b - \ol{b} = \eta_L(b) - \eta_R(b) \in \Cts(\Z_p \times \Z_p, \Z_p): (x,y) \mapsto x-y.	\]
In the notation of \Cref{lem: tensor product extended}, the primitives are included into $\Cts(\Z_p \times \Z_p, KO_*)$ via precomposition with
\[	m: \Z_p \times \Z_p \to \Z_p:\,\, m(x,y) = x-y.	\]
Thus, $b - \ol{b}$ is precisely $m^*(b)$. This proves that the map $\tmf_*\tmf \to KO_*KO$ sends $\ell$ to $b$. It follows that the map
\[	KO_* \otimes_{\tmf_*} \tmf_*\tmf \otimes_{\tmf_*} KO_* \to KO_*KO	\]
is an isomorphism.

Using the fiber sequence
\[	\tmf \stackrel{j^{-1}}{\to} \tmf \to KO,	\]
one obtains
\[	\tmf_*KO = \tmf_*\tmf/(\ol{j^{-1}}) = \tmf_*\tmf/(\ol{f}) = KO_*[f] \otimes \TT(\ell)/(f - \psi^p(\ell) + \ell).	\]
There is a pushout square of $L$-complete rings
\[	\xymatrix{ \TT(f) \ar[d] \ar[r]^{f \mapsto \psi^p(b) - b} & \TT(b) \ar[d]^{b \mapsto \ell} \\ KO_*[f] \ar[r] & KO_*[f] \otimes \TT(\ell)/(f - \psi^p(\ell) + \ell). \pushoutcorner }	\]
The top horizontal map is pro-free by \Cref{i injective flat}. By \Cref{lem: pro-free base change}, the bottom horizontal map is also pro-free.

Thus, the map of Hopf algebroids satisfies the hypotheses of the change-of-rings theorem, \Cref{prop: change of rings}, so induces an equivalence on Ext.
\end{proof}

\begin{cor}
The $K(1)$-local $\tmf$-based Adams spectral sequence for the sphere collapses at $E_2$, where it is concentrated on the 0 and 1 lines.
\end{cor}

\begin{proof}
As noted in \Cref{subsec: Hopf algebroids K KO}, we can identify Ext of $p$-complete $KO_*KO$-comodules with continuous group cohomology of $\Z_p^\times/\mu \cong \Z_p$. It is well-known that this profinite group has cohomological dimension 1. So the $E_2$ page of the spectral sequence is concentrated on the 0 and 1 lines and has no room for differentials.
\end{proof}

\appendix

\section{$\lambda$-rings and Hopf algebras} 
This section collects useful algebra related to the multiplicative theory of $K(1)$-local spectra. As we discuss in \Cref{subsec: theta-algebras}, any $K(1)$-local $E_\infty$-ring has power operations on its $\pi_0$ making it into a $\theta$-algebra (cf.~\cite{Bousfield96}, \cite{Rezk09}). We recall Bousfield's description of the free $\theta$-algebra functor and note that it takes values in Hopf algebras. In \Cref{subsec: lambda-rings}, we recall the definition of $\lambda$-rings, which are closely related to $\theta$-algebras -- see \cite{Bousfield96}. Unlike $\theta$-algebras, which are a vital feature of $K(1)$-local homotopy theory, $\lambda$-rings will largely play a technical role in some of the proofs in this paper. For this reason, we take the opportunity to clarify some of the ways of passing between $\lambda$-rings and $\theta$-algebras.

For the most part, we restrict to working with modules which are $p$-complete rather than merely $L$-complete. Let $\Mod_{\Z_p}^\wedge$ denote the category of $L$-complete $\Z_p$-modules, and recall from section \ref{sec: notation} that all algebraic statements carry a tacit completion.

\subsection{\texorpdfstring{$E_\infty$}{E infinity}-rings and \texorpdfstring{$\theta$}{theta}-algebras}\label{subsec: theta-algebras}
	
\begin{defn}
A \textbf{$\theta$-algebra} is an $L$-complete $\Z_p$-algebra $R$ equipped with operations $\theta:R \to R$
\begin{align*}
	\theta(x+y) &= \theta(x) + \theta(y) - \sum_{i=1}^{p-1} \frac{1}{p}\binom{p}{i}x^iy^{p-i}, & \\
	\theta(xy) &= x^p\theta(y) + y^p\theta(x) + p\theta(x)\theta(y) &\text{ for }x, y \in R_0,  \\
	\theta(1) &= 0. &
\end{align*}
\end{defn}
We will write $\psi^p(x) = x^p + p\theta(x)$ for $x$ in degree zero. Note that the above formulas imply that $\psi^p$ is a ring homomorphism in degree zero. Conversely, if $R$ is $p$-torsion-free, then $\theta$ can be uniquely recovered from a ring homomorphism $\psi^p$ satisfying $\psi^p(x) \equiv x^p$ mod $p$.

\begin{defn}\label{defn: psi-theta-algebra}
A \textbf{$\psi$-$\theta$-algebra} is a $p$-complete $\theta$-algebra $R$ together with maps $\psi^k:R \to R$ for $k \in \Z_p^\times$ such that
\begin{enumerate}
	\item $\psi^k$ is multiplicative on $R$,
	\item $k \mapsto \psi^k$ is a continuous endomorphism from $\Z_p^\times$ to the monoid of endomorphisms of $R_*$,
	\item and each $\psi^k$ commutes with $\theta$ and $\psi^p$.
\end{enumerate}
\end{defn}

\begin{prop}[{\cite[Chapter IX]{BMMS}, \cite{moduliproblems}}]
If $X$ is a $K(1)$-local $E_\infty$-ring spectrum such that $K_*X$ is $p$-complete, then $K_0X$ is naturally a $\psi$-$\theta$-algebra, with $\psi^k$ for $k \in \Z_p^\times$ given by the Adams operations.
\end{prop}

Since the Adams operations commute with the $\theta$-algebra structure, the $\theta$-algebra structure passes through the homotopy fixed points spectral sequence. Thus, if $X$ is a $K(1)$-local $E_\infty$-ring spectrum, $\pi_0 X$ is also a $\theta$-algebra. In other words, the classes in $K_0B\Sigma_p$ representing the power operations $\theta$ and $\psi^p$ lift to $\pi_0L_{K(1)}B\Sigma_p$ -- see \cite{K1localrings}.

\begin{prop}\label{K theta algebra}
The $\theta$-algebra structures on $\pi_0K = \pi_0KO = \Z_p$, on $KO_0KO$, and on $K_0K$ are all given by $\psi^p = \mathrm{id}$.
\end{prop}

\begin{proof}
There is a unique $\theta$-algebra structure on $\Z_p$ satisfying the requirements of \Cref{defn: psi-theta-algebra}, and it is $\psi^p = \mathrm{id}$.

As for $K_0K$ (the proof for $KO_0KO$ is similar), the multiplication map $K \sma K \to K$ is an $E_\infty$-map. By \Cref{thm: K Hopf algebroid}, the map induced on $\pi_0$ is
\[	
\Cts(\Z_p^\times, \Z_p) \ni f \mapsto f(1) \in \Z_p.
\]
Thus,
\[	
(\psi^p f)(1) = \psi^p(f(1)) = f(1).
\]
Moreover, $\psi^p$ commutes with the left action of the Adams operations, which act by 
\[	
(\psi^k \sma 1)(f)(x) = f(kx).
\]
It follows that
\[	
(\psi^p f)(k) = f(k)
\]
for every $k \in \Z_p^\times$. Thus, $\psi^p$ acts by the identity.
\end{proof}

There is an adjunction
\[
\begin{tikzcd}
	\TT: \mathsf{Mod}^\wedge_{\Z_p}\arrow[r, shift left] & \Alg_{\theta}\arrow[l, shift left]: U
\end{tikzcd}
\]
where $U$ is the forgetful functor, and $\TT$ is the free $\theta$-algebra functor. It is described explicitly as follows:

\begin{thm}[{\cite[2.6, 2.9]{Bousfield96}}]\label{thm: free theta-algebra}
The free $\theta$-algebra on a single generator $x$ is a polynomial algebra: explicitly,
\[	\TT(x) = \Z_p[x, \theta(x), \theta\theta(x), \dotsc]^\wedge_p \cong \Z_p[x,\theta_1(x), \theta_2(x), \dotsc]^\wedge_p,	\]
where the elements $\theta_n(x)$ are inductively defined so that
\[	\psi^{p^n}(x) = x^{p^n} + p\theta_1(x)^{p^{n-1}} + \dotsb + p^n\theta_n(x).	\]
\end{thm} 
	
Theta-algebras function as an algebraic approximation to $K(1)$-local $E_\infty$-algebras, as was shown in the following form by \cites{Rezk09, BarthelFrankland} following work of \cite{BMMS}. Write $\PP:\Sp \to \CAlg$ for the free $K(1)$-local $E_\infty$-algebra functor. This is given by
	\[
	\PP(X) = L_{K(1)}\left(\bigvee_{i\geq 0}E\Sigma_{i+}\wedge_{\Sigma_i}X^{\wedge i}\right).
	\]
	
	\begin{thm}[{\cite{BarthelFrankland}}]\label{thm: McClure}
		For a $K(1)$-local spectrum $X$, there is a natural map
		\[	\TT(K_*X) \to K_*(\PP(X)),	\]
		which is an isomorphism if $K_*X$ is flat as a $K_*$-module.
	\end{thm}

The category $\Alg_\theta$ has tensor products, which are the coproducts in this category. The tensor product of $R_1$ and $R_2$ has underlying ring $R_1 \otimes R_2$, and $\theta$-algebra structure
	\[
	\theta(x\otimes y) = x^p\otimes \theta(y)+ \theta(x)\otimes y^p+ p\theta(x)\otimes \theta(y).
	\]
If $R_1$ and $R_2$ are $\psi$-$\theta$-algebras, the tensor product has the same $\theta$-algebra structure as above, and has Adams operations
\[	\psi^k(x \otimes y) = \psi^k(x) \otimes \psi^k(y).	\]

Recall the adjunction
\[
	\begin{tikzcd}
	\TT: \Mod^\wedge_{\Z_p}\arrow[r, shift left] & \Alg_{\theta} \arrow[l, shift left] : U.
	\end{tikzcd}
\]
If the underlying module carries Adams operations, then the free $\theta$-algebra functor takes values in $\psi$-$\theta$-algebras. This yields an adjunction
	\[
	\begin{tikzcd}
	\TT: \Mod^\wedge_{\Z_p[[\Z_p^\times]]}\arrow[r, shift left] & \Alg_{\psi,\theta} \arrow[l, shift left] : U.
	\end{tikzcd}
	\]
Since $\TT$ is a left adjoint, it preserves coproducts. This results in the following natural isomorphism of $\theta$-algebras (resp. $\psi$-$\theta$-algebras),
	\[
	\TT(M\oplus N)\cong \TT(M)\,\otimes\, \TT(N).
	\]
In particular, this means that $\TT(M)$ has a natural $L$-complete Hopf algebra structure, with comultiplication coming from the diagonal map
\[	M \to M \oplus M.	\]
If $M$ is finitely generated and torsion-free, then $\TT(M)$ is actually a $p$-complete Hopf algebra. Moreover, the structure maps are morphisms of $\theta$-algebras (resp. $\psi$-$\theta$-algebras). 

	\begin{ex}\label{ex: Hopf algebra free theta algebra single generator}
		Of particular interest to us is the free $\theta$-algebra $\TT(b)$ on a single generator $b$. Its underlying algebra structure is given by \Cref{thm: free theta-algebra} above. An elementary calculation shows that $b$ is a Hopf algebra primitive, i.e. 
		\[
		\Delta(b) = b\otimes 1+1\otimes b. 
		\]
		Since $\Delta$ is a morphism of $\psi$-$\theta$ algebras, we have
		\[
		\psi^{p^n}\circ \Delta = \Delta \circ \psi^{p^n}.
		\]
		Since $\psi^{p^n}$ is a ring homomorphism for all $n$, we have 
		\[
		\Delta(\psi^{p^n}(b)) = \psi^{p^n}(b\otimes 1+1\otimes b) = \psi^{p^n}(b)\otimes 1 + 1\otimes \psi^{p^n}(b).
		\]
		Thus $\psi^{p^n}(b)$ is a Hopf algebra primitive for all $n$. This uniquely determines the rest of the Hopf algebra structure.
		
	This Hopf algebra is actually fairly classical. Recall that the additive group of $p$-typical Witt vectors of a $p$-complete ring $R$ are classified by a Hopf algebra
	\[ \Witt = \Z_p[a_0,a_1,\dotsc].	\]
	The map that sends $\theta_n(b)$ to $a_n$ is then an isomorphism $\TT(b) \to \Witt$ of $\theta$-algebras and Hopf algebras. The element $\psi^{p^n}(b)$ goes to the primitive element of $\Witt$,
	\[	w_n = a_0^{p^n} + pa_1^{p^{n-1}} + \dotsb + p^na_n,	\]
	which represents the $n$th ghost component.
	\end{ex}

\subsection{\texorpdfstring{$\lambda$-rings}{Lambda-rings}}\label{subsec: lambda-rings}

\begin{defn}[\cite{Bousfield96, Wilkerson_1982}]\label{def: lambda-rings}
	A \emph{$\lambda$-ring} is a graded commutative $p$-complete $\Z_p$-algebra $R$ equipped with operaitons $\lambda^n: R\to R$ for $n\geq 0$ such that 
	\begin{align*}
		\lambda^0(x) &=1,\\
		\lambda^1(x) &= x, \\
		\lambda^n(1) &=0 \text{ for } n\geq 1, \\
		\lambda^n(x+y) &= \sum_{i+j=n}\lambda^i(x)\lambda^j(y)\\
		\lambda^n(xy) &= P_n(\lambda^1(x), \ldots, \lambda^n(x), \lambda^1(y), \ldots, \lambda^n(y)), \text{ and }\\
		\lambda^m(\lambda^n(x)) &= P_{m,n}(\lambda^1(x), \ldots, \lambda^{mn}(x)).
	\end{align*}
	where $P_n$ and $P_{m,n}$ are certain universal polynomials with integral coefficients which can be recovered by taking $\lambda^n$ to be the $n$th elementary symmetric polynomial in infinitely many variables.
\end{defn}

The category $\Alg_\lambda$ of $\lambda$-rings is also symmetric monoidal. The tensor product is the ordinary $p$-complete tensor product with $\lambda$-operations defined by the Cartan formula, 
\[
\lambda^n = \sum_{i+j=n}\lambda^i\otimes \lambda^j.
\]

The notions of a $\lambda$-ring and a $\psi$-$\theta$-algebra are closely related. In particular, given a $\lambda$-ring we can associate to it Adams operations. Indeed, one defines
\[
\psi^n(x) = \nu_n(\lambda^1(x), \ldots, \lambda^n(x)).
\]
Here, $\nu_n$ is the polynomial so that if $\sigma_k$ denotes the $k$th elementary symmetric polynomial in infinitely many variables $x_i$ and $p_k = \sum x_i^k$, 
\[
p_n(\underline{x}) = \nu_n(\sigma_1(\underline{x}), \ldots \sigma_n(\underline{x})).
\]
The operation $\psi^p$ satisfies the Frobenius congruence $\psi^p(x) \equiv x^p\mod p$. Thus if $R$ is a torsion-free $p$-complete $\lambda$-ring, then $R$ is a $\psi$-$\theta$-algebra. A partial converse also holds. 

\begin{thm}[{Bousfield, \cite[Theorem 3.6]{Bousfield96}}]\label{thm: Bousfield lambda rings}
	A $p$-complete $\psi$-$\theta$-algebra has a unique structure as $\lambda$-ring whose Adams operations are the given $\psi^k$ and $\psi^p$. 
\end{thm}

\begin{defn}\label{lambda-ring functors}
As a result, there are not one but two functors from $\psi$-$\theta$-algebras to $\lambda$-rings, both of which are the identity on underlying rings. The \emph{sealed} functor,
\[	\mathcal{S}:\Alg_{\psi,\theta} \to \Alg_\lambda,	\]
is the one given by Bousfield's theorem, and is an equivalence on the subcategories of torsion-free algebras. The \emph{leaky} functor,
\[	\mathcal{L}:\Alg_{\psi,\theta} \to \Alg_\lambda,	\]
first replaces all the $\psi^k$ by the identity for $k$ prime to $p$, and then applies $\mathcal{S}$ to the result.

We will also write $\mathcal{L}$ for the functor
\[	\Alg_{\theta} \to \Alg_\lambda	\]
which sets $\psi^k = 1$ for $k$ prime to $p$ and then applies $\mathcal{S}$ to the result.
\end{defn}

\begin{ex}\label{Zp lambda-ring}
Recall that $\Z_p$ has a unique $\theta$-algebra structure, in which $\psi^p$ is the identity. Thus $\mathcal{L}(\Z_p)$ is a $\lambda$-ring in which all Adams operations are the identity. The $\lambda$-operations are given by $\lambda^n(x) = \binom{x}{n}$ \cite[Example 1.3]{Bousfield96}. 
\end{ex}

\begin{lem}
Both $\mathcal{S}$ and $\mathcal{L}$ are symmetric monoidal functors.
\end{lem}

\begin{proof}
As the operation of replacing the prime-to-$p$ Adams operations with the identity is clearly monoidal, it suffices to prove that $\mathcal{S}$ is monoidal. For this, it suffices to prove that the inverse operation, from $\lambda$-rings to rings with Adams operations $\psi^n$ for $n \in \Z_p$, preserves the obvious tensor products, which is a simple calculation.
\end{proof}

\begin{cor}\label{lem: coproduct lambda algebras}
		Let $M$ be a torsion-free, $p$-complete $\Z_p$-module. Then the coproduct map 
		\[
		\Delta:\mathcal{L}(\TT(M))\to \mathcal{L}(\TT(M))\otimes \mathcal{L}(\TT(M))
		\]
		is a morphism of $\lambda$-rings.
\end{cor}

\bibliographystyle{plain}
\bibliography{K(1)tmfcoops}
\end{document}